\documentclass[11pt]{amsart}

\allowdisplaybreaks[2] \textwidth16cm \textheight22cm \headheight12pt \oddsidemargin 0cm
\newcommand{\mcD}{\mathcal{D}}
\newcommand{\mcE}{\mathcal{E}}
\newcommand{\mcCh}{\mathcal{C}h}
\newcommand{\mcOv}{\mathcal{O}v}
\newcommand{\mcBr}{\mathcal{B}r}
\newcommand{\mcFr}{\mathcal{F}r}
\DeclareMathOperator{\mfS}{\mathfrak{S}}
\newcommand{\scp}{\operatorname{SCP}}
\newcommand{\DyckP}{\operatorname{DyckP}}

\usepackage{amssymb}
\usepackage{tikz, graphicx, ifthen}
\usepackage{verbatim}
\usepackage{xcolor}
\usepackage[nomessages]{fp}
\usepackage{mathtools}
\usetikzlibrary{calc}
\usetikzlibrary{decorations.pathreplacing}
\usepackage{mathdots}

\usepackage{hyperref}

\usepackage[nameinlink]{cleveref}
\newtheorem{theorem}{Theorem}[section]
\newtheorem{thm}[theorem]{Theorem}
\newtheorem{lemma}[theorem]{Lemma}
\newtheorem{proposition}[theorem]{Proposition}

\newtheorem{example}[theorem]{Example}
\newtheorem{definition}[theorem]{Definition}

\newtheorem{remark}[theorem]{Remark}
\newtheorem{corollary}[theorem]{Corollary}
\newtheorem{cor}[theorem]{Corollary}
\newtheorem{conjecture}[theorem]{Conjecture}

\definecolor{red}{rgb}{1,0,0}
\definecolor{blue}{rgb}{0,0,1}
\definecolor{light-gray}{gray}{.9}

\newcounter{x}
\newcounter{y}

\newcommand*\cubecolors[1]{%
  \ifcase#1\relax
  \or\colorlet{cubecolor}{cyan}%
  \or\colorlet{cubecolor}{green}%
  \or\colorlet{cubecolor}{yellow}%
  \or\colorlet{cubecolor}{pink}%
  \or\colorlet{cubecolor}{orange}%
  \or\colorlet{cubecolor}{purple}%
  \or\colorlet{cubecolor}{white}%
  \else
    \colorlet{cubecolor}{black}%
  \fi
}
\newcommand\yaxis{180}
\newcommand\zaxis{-27}
\newcommand\xaxis{90}

\newcommand\topside[3]{
  \fill[fill=cubecolor, draw=black,shift={(\xaxis:#1)},shift={(\yaxis:#2)},
  shift={(\zaxis:#3)}] (0,0) -- (1,0) -- (0.5,0.25) --(-0.5,0.25)--(0,0);
}

\newcommand\leftside[3]{
  \fill[fill=cubecolor, draw=black,shift={(\xaxis:#1)},shift={(\yaxis:#2)},
  shift={(\zaxis:#3)}] (0,0) -- (0,-1) -- (-0.5,-0.75) --(-0.5,0.25)--(0,0);
}

\newcommand\rightside[3]{
  \fill[fill=cubecolor, draw=black,shift={(\xaxis:#1)},shift={(\yaxis:#2)},
  shift={(\zaxis:#3)}] (0,0) -- (1,0) -- (1,-1) --(0,-1)--(0,0);
}

\newcommand\cube[3]{
  \topside{#1}{#2}{#3} \leftside{#1}{#2}{#3} \rightside{#1}{#2}{#3}
}

\newcommand\ppAff[2]{
 \setcounter{x}{0}
 \foreach \a in {#2} {
    \addtocounter{x}{1}
    \setcounter{y}{-1}
    \foreach \b in \a {
      \addtocounter{y}{1}
      \ifthenelse{\b=0}{\addtocounter{y}{0}}{
        \FPeval{\result}{clip(#1-\the\value{y}-1)}
        \cubecolors{\b}
        \cube{\value{x}}{\value{y}}{0};
        \FPeval{\result}{clip(#1-\the\value{y}-1)}
        \draw[draw=black,shift={(\xaxis:\value{x})},shift={(\yaxis:\value{y})},
  shift={(\zaxis:0)}] (0.5,-0.5) node {\textsf{\result}};}
    }
  }
}

\newcommand\ppsansnumbers[2]{
 \setcounter{x}{0}
 \foreach \a in {#2} {
    \addtocounter{x}{1}
    \setcounter{y}{-1}
    \foreach \b in \a {
      \addtocounter{y}{1}
      \ifthenelse{\b=0}{\addtocounter{y}{0}}{
        \FPeval{\result}{clip(#1-\the\value{y}-1)}
        \cubecolors{\b}
        \cube{\value{x}}{\value{y}}{0};
        \FPeval{\result}{clip(#1-\the\value{y}-1)}
        \draw[draw=black,shift={(\xaxis:\value{x})},shift={(\yaxis:\value{y})},
  shift={(\zaxis:0)}] (0.5,-0.5) node { };}
    }
  }
}

\newcommand\topsidedashed[3]{
  \fill[fill=cubecolor, draw=black,shift={(\xaxis:#1)},shift={(\yaxis:#2)},
  shift={(\zaxis:#3)}] [dashed](0,0) -- (1,0) -- (0.5,0.25) --(-0.5,0.25)--(0,0);
}

\newcommand\leftsidedashed[3]{
  \fill[fill=cubecolor, draw=black,shift={(\xaxis:#1)},shift={(\yaxis:#2)},
  shift={(\zaxis:#3)}] [dashed](0,0) -- (0,-1) -- (-0.5,-0.75) --(-0.5,0.25)--(0,0);
}

\newcommand\rightsidedashed[3]{
  \fill[fill=cubecolor, draw=black,shift={(\xaxis:#1)},shift={(\yaxis:#2)},
  shift={(\zaxis:#3)}] [dashed](0,0) -- (1,0) -- (1,-1) --(0,-1) -- (0,0);
}

\newcommand\cubedashed[3]{
  \topsidedashed{#1}{#2}{#3} \leftsidedashed{#1}{#2}{#3} \rightsidedashed{#1}{#2}{#3}
}

\newcommand\ppsansnumbersdashed[2]{
 \setcounter{x}{0}
 \foreach \a in {#2} {
    \addtocounter{x}{1}
    \setcounter{y}{-1}
    \foreach \b in \a {
      \addtocounter{y}{1}
      \ifthenelse{\b=0}{\addtocounter{y}{0}}{
        \FPeval{\result}{clip(#1-\the\value{y}-1)}
        \cubecolors{\b}
        \ifthenelse{\result=0\OR\value{y}=0}
        {
        \cubedashed{\value{x}}{\value{y}}{0};
        \FPeval{\result}{clip(#1-\the\value{y}-1)}
        \draw[draw=black,shift={(\xaxis:\value{x})},shift={(\yaxis:\value{y})}, shift={(\zaxis:0)}] (0.5,-0.5) node {};}
        {
        \cube{\value{x}}{\value{y}}{0};
        \FPeval{\result}{clip(#1-\the\value{y}-1)}
        \draw[draw=black,shift={(\xaxis:\value{x})},shift={(\yaxis:\value{y})},shift={(\zaxis:0)}] (0.5,-0.5) node { };}}
    }
  }
}

\usepackage{lipsum}
\usepackage{tikz, graphicx, ifthen}
\usepackage{cleveref}
\title[Staircase diagrams and smooth permutations]{Staircase diagrams, pattern avoidance, and smooth permutations}

\author{Faqruddin Ali Azam}\address{Math \& Science Division\\ Lewis-Clark State College  \\ Lewiston, Idaho 83501 \\ U.S.A. }\email{sazam@lcsc.edu}

\author{Edward Richmond}\address{Department of Mathematics\\ Oklahoma State University  \\ Stillwater, Oklahoma 74078 \\ U.S.A. }\email{edward.richmond@okstate.edu}

\begin{document}

\maketitle

\begin{abstract}A well known result due to Lakshmibai and Sandhya states that smooth Schubert varieties of type A correspond to permutations that avoid the patterns 3412 and 4231. It was later shown by the second author and Slofstra that staircase diagrams over a simply-laced Dynkin diagram are in bijection with smooth Schubert varieties of the corresponding type. We explore how the poset structure of staircase diagrams is connected with pattern avoidance on smooth permutations. As an application, we enumerate several subclasses of smooth permutations which are characterized by pattern avoidance. These subclasses include the class of polished permutations which were studied by Gaetz and Gao.
\end{abstract}


\section{Introduction}

Let $\mfS_{n}$ denote the symmetric group of permutations on $[n]:=\{1,2,\ldots, n\}.$  One important collection of permutations are \textbf{smooth permutations} which correspond to smooth Schubert varieties of type $A$.  In \cite{LS90}, Lakshmibai and Sandhya prove that smooth Schubert varieties are classified by permutation pattern avoidance.  In \cite{Ca94}, Carrell proves that smooth Schubert varieties are also characterized by rank symmetry of the Bruhat interval of the corresponding permutation.  Together, these results give the following theorem on smooth permutations. 

\begin{theorem} [\cite{Ca94, LS90}] Let $w\in \mfS_{n}$.  Then the following are equivalent:
\begin{enumerate}
    \item The Schubert variety of $w$ is smooth.
    \item The permutation $w$ avoids patterns $3412$ and $4231$.
    \item The Bruhat interval $[e,w]$ is rank symmetric.
\end{enumerate}
\end{theorem}

In \cite{Ry87}, Ryan proved that smooth Schubert varieties of type A are iterated fiber-bundles of Grassmannian varieties.  In \cite{Ha92}, Haiman uses Ryan's fiber-bundle structure to calculate the generating function for smooth permutations (this generating function also appears in \cite{BMB07}).   

\begin{theorem}[\cite{BMB07, Ha92}]\label{thm:smooth_enumeration}
For $n\geq 0$, let $a_n:=\#\{w\in \mfS_{n+1}\ |\ \text{$w$ avoids $3412$ and $4231$}\}$.  Then the generating function of the sequence $a_n$ is 
$$\sum_{n\geq 0} a_n\,x^n=\frac{1-5x+4x^2+x\sqrt{1-4x}}{1-6x+8x^2-4x^3}.$$
\end{theorem}

The goal of this paper is to calculate generating functions for certain subclasses of smooth permutations which are characterized by pattern avoidance.  One such subclass are polished permutations.  We say a permutation $w$ is \textbf{polished} if the corresponding Bruhat interval $[e,w]$ is self-dual as a poset.  Since self-duality implies rank symmetry, polished permutations form a subclass of smooth permutations.  The following is proved by Gaetz and Gao in \cite{GG20}.

\begin{theorem}[\cite{GG20}]\label{thm:GG_polished}
Let $w\in \mfS_{n}$.  Then the following are equivalent:
\begin{enumerate}
\item The permutation $w$ is polished (i.e. $[e,w]$ is self-dual).
\item The permutation $w$ is smooth and avoids patterns $34521$, $54123$, $45321$, and $54312$.
\end{enumerate}
\end{theorem}
Our first result is a calculation of the generating function for polished permutations.
\begin{theorem}\label{thm:polished_GF}
Let $p_n:=\#\{w\in \mfS_{n+1}\ |\ \text{$w$ is polished}\}$.  Then
\begin{equation}\label{eqn:polished_GF}
\sum_{n\geq 0} p_n\,x^n=\frac{1-3x+x^2}{1-5x+5x^2-2x^3}.
\end{equation}

\end{theorem}

The tool we use to prove \Cref{thm:polished_GF} are \textbf{staircase diagrams} (of type A) over the set of integers $[n]$.  Staircase diagrams were first studied in \cite{RS17} and provide a combinatorial model for fiber-bundle structures on smooth and rationally smooth Schubert varieties of any finite Lie-type.  We show that there is a correspondence between imposing combinatorial conditions on staircase diagrams and imposing pattern avoidance conditions on smooth permutations (see \Cref{thm:strong_chains} and \Cref{thm:overlaps}).  The advantage of this correspondence is that staircase diagrams exhibit many structures which can be used in enumeration that are not apparent from a pattern avoidance perspective.  As an application, we calculate the following two additional generating functions.

\begin{theorem}\label{thm:semi-polished_GF12}
Let $b_n:=\#\{w\in \mfS_{n+1}\ |\ w\text{ avoids $3412$, $4231$, $34521$,\text{ and } $54123$}\}$.  Then 
\begin{equation}\label{eqn:polished_type1_GF}
\sum_{n\geq 0}b_n\, x^n=\frac{1-5x+6x^2-4x^3}{1-6x+10x^2-8x^3+2x^4}.
\end{equation}
Let  $c_n:=\#\{w\in \mfS_{n+1}\ |\ w\text{ avoids $3412$, $4231$, $45321$,\text{ and } $54312$}\}$.  Then 
\begin{equation}\label{eqn:polished_type2_GF}
\sum_{n\geq 0}c_n\, x^n=\frac{1-4x+x^2}{1-5x+4x^2}.
\end{equation}
\end{theorem}
One interesting consequence of \Cref{eqn:polished_type2_GF} is the following corollary.
\begin{cor}\label{cor:d_n_is_c_n}
    Let $d_n:=\#\{w\in \mfS_{n+1}\ |\ w\text{ avoids $4123$ and $4321$}\}$.  Then $d_n=c_n$ from \Cref{thm:semi-polished_GF12}.
\end{cor}
The proof of \Cref{cor:d_n_is_c_n} is due to the fact that the sequences $c_n$ and $d_n$ share the same generating function found in \Cref{eqn:polished_type2_GF}.  A formula for $d_n$ is calculated by Kremer and Shiu in \cite{KS03}.  Recently, Bridges and Waite construct a direct bijection between these two sets of permutations in \cite{BW26}.

The permutation patterns found in \Cref{thm:GG_polished} and \Cref{thm:semi-polished_GF12} are especially well-behaved with respect to the structure of staircase diagrams.  One of our main goals is to characterize which permutations exhibit this type of compatibility.  As a result, we define and study two special types of permutations which are subclasses of indecomposable permutations.  The first class we call \textbf{strongly indecomposable permutations} (\Cref{def:wk_decompose}) and the second we call \textbf{chain indecomposable permutations} (\Cref{def:chain-indecomposable}).  We will see that the permutations found in \Cref{thm:GG_polished} and \Cref{thm:semi-polished_GF12} are both strong and chain indecomposable.  Moreover, they are the smallest examples in a larger family of permutations whose avoidance can be studied using staircase diagrams.  We also prove that the generating functions corresponding to avoiding these families of permutations are rational, similar to \Cref{thm:GG_polished} and \Cref{thm:semi-polished_GF12}.

The structure of this paper is as follows.  In \Cref{sec:decomp_perms}, we define the notion of a strongly indecomposable permutation and give some connections with pattern avoidance.  In \Cref{sec:staircase_diagrams} and \Cref{sec:sd_patterns}, we discuss staircase diagrams and how staircase poset structures correspond to pattern avoidance.  In \Cref{sec:GF_calc}, we define the notion of a chain indecomposable permutation and apply the techniques we develop in this paper to calculate several generating functions.  In particular, we prove \Cref{thm:polished_GF} and \Cref{thm:semi-polished_GF12}.  We show that many generating function calculations can be reduced to enumerative problems on Catalan objects such as Dyck paths that are bounded by various conditions. 

\subsection*{Acknowledgments} The authors would like to thank Mikl\'{o}s B\'{o}na and Alex Woo for helpful discussions.  The second author was supported by a grant from the Simons Foundation 941273.

\subsection*{AI disclosure} We used AI to assist us in the proofs of \Cref{cor:chain_rational} and \Cref{cor:overlap_rational}.  In particular, we asked Claude Opus 5.5 about the rationality of the generating function for the number of Dyck paths with certain constraints on the number of peaks and valley heights.  For \Cref{cor:chain_rational}, AI produced a correct formula for the generating function in question, but we eventually found a direct reference for this formula, so we cite these published results.  For \Cref{cor:overlap_rational}, AI  correctly determined that the generating functions in question are rational and produced a short argument.  We later found a reference that implies the rationality of these generating functions, but still included a proof (slightly modified from the AI proof) since the reference paper works in a much larger category of generating functions related to Dyck paths and does not mention our special case.  We did not use AI to assist us with any other results in this paper.  We also did not use AI to assist us in writing this paper.

\section{Strongly indecomposable permutations}\label{sec:decomp_perms}

In this section, we discuss decomposability of permutations in relation to pattern avoidance.  Let $w\in \mfS_n$ and $p\in \mfS_r$ for some $r\leq n$ and write $w=w(1)w(2)\cdots w(n)$ and $p=p(1)p(2)\cdots p(r)$ in one-line notation.  We say $w$ \textbf{contains the pattern} $p$ if there is a subsequence $(i_1<\cdots <i_r)$ such that $w(i_1)\cdots w(i_r)$ has the same relative order as $p(1)\cdots p(r)$.  Otherwise, we say $w$ \textbf{avoids the pattern} $p$. 


The symmetric group $\mfS_n$ is a Coxeter group generated by $S=\{s_1,s_2,\cdots,s_{n-1}\}$, where $s_i$ denotes the simple transposition $(i,i+1)$.  Let $\leq$ denote the Bruhat order on $\mfS_n$.  For $w\in \mfS_n$, define the \textbf{support set}
$$S(w):=\{i\in [n-1]\ |\ s_i\leq w\}.$$
In other words, $k\in S(w)$ if and only if $s_k$ appears in some (or all) expression of $w$ in terms of simple transpositions.
We say $w$ is \textbf{decomposable} in $\mfS_n$ if there exists $1\leq i<n$ such that 
$$\{w(1),\ldots, w(i)\}=[i].$$
If $w$ is not decomposable, then it is \textbf{indecomposable}.  We summarize the connection between decomposability, support sets, and pattern avoidance in the following lemma.
\begin{lemma}\label{lem:decomposable}
The permutation $w$ is decomposable in $\mfS_n$ if and only if there exists $1\leq i<n$ and a decomposition $w=w_1\cdot w_2$ such that $S(w_1)\subseteq [i-1]$ and $S(w_2)\subseteq [n-1]\setminus [i]$.

Furthermore, if $w=w_1\cdot w_2$ is decomposable in $\mfS_n$ as above, and $p$ is indecomposable in $\mfS_r$ for some $r\leq n$, then $w$ avoids $p$ if and only if $w_1$ and $w_2$ avoid $p$.
\end{lemma}

We call the factorization $w=w_1\cdot w_2$ in \Cref{lem:decomposable} a \textbf{decomposing factorization}.  For more background on decomposable/indecomposable permutations, see \cite{Ki11}.
 We define an analogue of the decomposable property  called \textbf{weakly-decomposable}. 
\begin{definition}\label{def:wk_decompose}
    We say a permutation $w$ is \textbf{left-weakly decomposable} in $\mfS_n$ if there exists an integer $1\leq i< n-1$ such that
    \[\{w(1),\ldots, w(i)\}\subseteq [i+1].\]
    A permutation $w$ is \textbf{weakly decomposable} if either $w$ or $w^{-1}$ is left-weakly decomposable.  
    If $w$ is not weakly decomposable, we say that $w$ is \textbf{strongly indecomposable}.  
\end{definition}

Note that any decomposable permutation is also weakly-decomposable.  The next lemma is an analogue of \Cref{lem:decomposable} in regard to support sets.

\begin{lemma}\label{lem:wk_decomposable_support}
A permutation $w$ is left-weakly decomposable in $\mfS_n$ if and only if there exist $1\leq i<n-1$ and a unique decomposition $w=w_1\cdot w_2$ such that $S(w_1)\subseteq [i]$ and $S(w_2)\subseteq [n-1]\setminus [i]$.
\end{lemma}

\begin{proof}
Suppose we have permutations $w_1$ and $w_2$ with $S(w_1)\subseteq [i]$ and $S(w_2)\subseteq [n-1]\setminus[i]$ for some $1 \leq i < n-1$.  Then for any $k\leq i$, we have
$w_1(k)\leq i+1$ and $w_2(k)=k,$ which implies that $w_1\cdot w_2(k)\leq i+1$.  This implies that $w_1\cdot w_2$ is left-weakly decomposable.  

Conversely, suppose that $w$ is left-weakly decomposable.  Then there exists $1\leq  i< n-1$ such that $w(k)\leq i+1$ for all $k\leq i$.  Define 
\[x=\min\{w(i+1),\ldots,w(n)\}\] and $j=w^{-1}(x)$.  Note that $x\leq i+1$ and $j\geq i+1$.  Define the permutations $w_1$ and $w_2$ by 
\begin{equation}\label{eqn:weak_facotorization}
w_1(k):=\begin{cases}
    w(k) & \text{if $k\leq i$}\\
    x & \text{if $k=i+1$}\\
    k & \text{if $k>i+1$}
\end{cases}\quad \text{and }\quad 
w_2(k):=\begin{cases}
    k & \text{if $k\leq i$}\\
    i+1& \text{if $k=j$}\\
    w(k) & \text{if $k> i+1$ and $k\neq j$}
\end{cases}
\end{equation}
Since $w_1(k)=k$ for all $k> i+1$ and we have $S(w_1)\subseteq [i]$.  Furthermore, since $w_2(k)=k$ for all $k\leq i$, we have $S(w_2)\subseteq [n]\setminus[i]$.  Finally, it is easy to verify that $w=w_1\cdot w_2$.  For uniqueness, we observe that $w=w_1\cdot w_2$ is a parabolic decomposition with respect to $S(w_2)$ and hence unique \cite[Proposition 2.4.4]{BB05}.  This completes the proof.
\end{proof}

Unlike decomposing factorizations, the weakly decomposing factors $w_1$ and $w_2$ found in \Cref{lem:wk_decomposable_support} do not necessarily commute.  However, \Cref{lem:wk_decomposable_support} implies that if $w$ is weakly decomposable, then there is a factorization $w=w_1\cdot w_2$ such that either 
\begin{itemize}
    \item $S(w_1)\subseteq [i]$ and $S(w_2)\subseteq [n-1]\setminus[i]$ or
    \item $S(w_1)\subseteq [n-1]\setminus[i]$ and $S(w_2)\subseteq [i]$ for some $1\leq i<n-1$.
\end{itemize} 
We call such a factorization of $w$ a \textbf{weakly-decomposing factorization}.  For a permutation $w\in S_{n}$, we represent the corresponding permutation matrix of $w$ with a {``\large$\bullet$"} in $(i,w(i))$ position where $(1,1)$ represents the northwest corner.

\begin{example}Consider the permutations $w=4152736$ and $w'=3714256$ whose matrices appear below (respectively).

$$
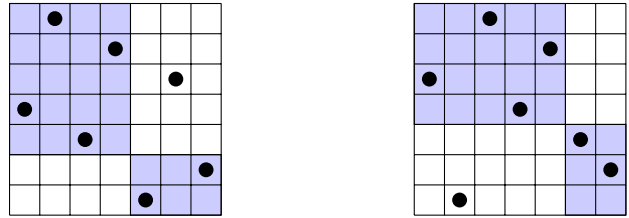
\begin{tikzpicture}[scale=0.4]
\draw[fill=blue!20] (0,7) rectangle (4,2);
\draw[fill=blue!20] (4,0) rectangle (7,2);
\draw[step=1.0,black] (0,0) grid (7,7);
\fill (0.5,3.5) circle (7pt);
\fill (1.5,6.5) circle (7pt);
\fill (2.5,2.5) circle (7pt);
\fill (3.5,5.5) circle (7pt);
\fill (4.5,0.5) circle (7pt);
\fill (5.5,4.5) circle (7pt);
\fill (6.5,1.5) circle (7pt);
\end{tikzpicture}\hspace{1in}
\begin{tikzpicture}[scale=0.4]
\draw[fill=blue!20] (0,7) rectangle (5,3);
\draw[fill=blue!20] (5,0) rectangle (7,3);
\draw[step=1.0,black] (0,0) grid (7,7);
\fill (5.5,2.5) circle (7pt);
\fill (6.5,1.5) circle (7pt);
\fill (1.5,0.5) circle (7pt);
\fill (3.5,3.5) circle (7pt);
\fill (0.5,4.5) circle (7pt);
\fill (4.5,5.5) circle (7pt);
\fill (2.5,6.5) circle (7pt);
\end{tikzpicture}$$
Note that $w$ and $(w')^{-1}$ are both left-weakly decomposable with respect to $i=4$ and hence $w$, $w'$ are both weakly decomposable.  We highlight in blue how their permutation matrices are ``almost" block diagonal.  The corresponding weakly decomposing factorizations are
$$w=(s_3s_4s_2s_3s_1)\cdot(s_6s_5)\quad\text{and}\quad w'=(s_6s_5)\cdot(s_4s_2s_3s_4s_1s_2).$$

\end{example}

The combination of \Cref{def:wk_decompose} and \Cref{lem:wk_decomposable_support} give a simple way to check if a permutation has a weakly-decomposing factorization by only looking at the one-line notation.  The next lemma generalizes the second part of \Cref{lem:decomposable}.  

\begin{lemma}\label{lem:wk_decomposable_patterns}
Let $w=w_1\cdot w_2$ be a weakly decomposing factorization in $\mfS_n$.  If $p$ is strongly indecomposable in $\mfS_r$ for some $r\leq n$, then $w$ avoids $p$ if and only if $w_1$ and $w_2$ avoid $p$.
\end{lemma}

\begin{proof}
    First note that $p$ is strongly indecomposable if and only if $p^{-1}$ is strongly indecomposable.  Moreover, $w$ contains $p$ if and only if $w^{-1}$ contains $p^{-1}$.  Hence, without loss of generality, we may assume that $w$ is left-weakly decomposable.  
    
    Let $w=w_1\cdot w_2$ is a left-weakly decomposing factorization such that $S(w_1)\subseteq [i]$ and $S(w_2)\subseteq[n]\setminus[i]$.  Suppose $w$ avoids $p$.  For the sake of contradiction, if $w_1$ contains $p$, then any containment of $p$ must be within the sequence $w_1(1)\cdots w_1(i+1)$ since $w_1(k)=k$ for $k>i+1$ and $p$ is strongly indecomposable.  But $w$ contains the sequence $w_1(1)\cdots w_1(i+1)$ and hence would also contain $p,$ which is a contradiction.  Thus $w_1$ avoids $p$.  By a similar argument, $w_2$ also avoids $p$.  

    Conversely, assume that $w_1$ and $w_2$ each avoid $p$ and suppose that $w$ contains $p$.  If $p$ is contained within $w(1)\cdots w(i)$, then $w_1$ would contain $p$.  Similarly, if $p$ is contained within $w(i+1)\cdots w(n)$, then $w_2$ would contain $p$. Hence any containment of $p$ within $w$ must involve entries before and after $w(i+1)$ in one-line notation.
    
    Let $p=p(1)\cdots p(r)$ and fix a containment of $p$ within $w$.   Let $p(1)\cdots p(r')$ denote the sequence entries contained within $w(1)\cdots w(i)$.  Note that this sequence is non-empty and $r'<r$.     
    Since $\{w(1),\ldots, w(i)\}\subseteq [i+1]$, we must have $\{p(1),\ldots,p(r')\}\subseteq [r'+1]$.  This implies that $p$ is left-weakly decomposable, which is a contradiction.  Thus $w$ must avoid $p$.
\end{proof}

\section{Staircase diagrams}\label{sec:staircase_diagrams}
In this section, we define staircase diagrams and review several properties from \cite{RS17}.  For a subset $B\subseteq [n-1]=\{1,2,\ldots,n-1\}$, we use the terminology \textbf{connected} and \textbf{disconnected} to characterize whether or not $B$ is an interval of consecutive integers.

\begin{definition}\label{def:staircase}  Let $\mcD=\{B_1,\ldots,B_k\}$ be a set of non-empty intervals in $[n-1]$ where $B_i=[l_i,r_i]$ and
\[\text{$l_i<l_{i+1}$ and $r_i<r_{i+1}$  for all  $i$.}\]
We say a partial order $\preceq$ on $ \mcD$ is a \textbf{staircase diagram} on $[n-1]$ if the following conditions hold:
\begin{enumerate}
\item $(B_i,B_j)$ is a cover relation if and only if $j=i\pm 1$ and $B_i\cup B_j$ is connected.
\item If $B_{i-1}\succ B_i\prec B_{i-1}$ or $B_{i-1}\prec B_i\succ B_{i+1}$, then $B_{i-1}\cup B_{i+1}$ is disconnected.
\end{enumerate}
We also include $\mcD=\varnothing$ as the empty staircase diagram.
\end{definition}
We call the intervals $B_i$ \textbf{blocks} of $\mcD$ and the elements of a block, \textbf{boxes}. We can represent a staircase diagram with a picture of ``stairs" of varying sizes and with steps up/down corresponding to cover relations as seen in \Cref{Ex:typeA_11.1}. 
\begin{example}\label{Ex:typeA_11.1}
The picture  
\begin{equation*}
\begin{tikzpicture}[scale=0.45]
            \ppAff{12}{ {0,0,0,0,0,0,0,0,1,1,1},{4,0,0,2,2,0,0,3,3,3,0},{0,6,6,0,5,5,5,5,0,0,0}}
    \end{tikzpicture}
 \end{equation*} represents the staircase diagram $\mcD =\{[1,3] \prec [2,4] \prec [4,7]\succ  [7,8]     \prec [9,10]\succ \{11\}\}$.
\end{example}

 We assign a permutation to a staircase diagram as follows.  Recall that the symmetric group $\mfS_n$ is generated by $S=\{s_1,s_2,\ldots,s_{n-1}\}$ where $s_i=(i,i+1)$.   For any $J\subseteq [n-1]$, let $u_J$ denote the \textbf{longest permutation} generated by $\{s_i \ |\ i\in J\}$.  If $J=\{i\}$ is a singleton, then we write $u_i:=u_{\{i\}}=s_i$.  For $B\in \mcD$, define 
\[J(B):=B \cap \left(\bigcup_{B'\prec B} B'\right).\]
In other words, $J(B)$ is the set of boxes in $B$ that also belong to blocks below $B$ in the picture of $\mcD$. If $B$ is minimal in $\mcD$, then $J(B)=\varnothing$ and set $u_{\varnothing}=\text{id}$.
\begin{definition}\label{def:staircase_permutation}
Let $\mcD$ be a staircase diagram on $[n-1]$.  We define the permutation $\Lambda(\mcD)$ inductively on the number of blocks.  First, set $\Lambda(\varnothing):=\text{id}$.  Otherwise, choose a maximal block $B\in\mcD$ and define
$$\Lambda(\mcD):=u_B\cdot u_{J(B)}\cdot \Lambda(\mcD\setminus\{B\}).$$
\end{definition}

\begin{example}
Let $n=9$ and $\mcD=\{[1,3]\prec [2,4]\prec [4,7]\succ [7,8]\}$. 
Then \\
 
 \begin{align*}
 \Lambda\left(\begin{tikzpicture}[scale=0.45, baseline={(0,0.5)}]
            \ppAff{9}{ {0,0,0,0,0,1,1,1},{2,2,0,0,3,3,3,0},{0,5,5,5,5,0,0,0}}
    \end{tikzpicture}\right)    
    &=(u_{[4,7]}\cdot u_4u_7)\cdot \Lambda\left(\begin{tikzpicture}[scale=0.45, baseline={(0,0.5)}]
            \ppAff{9}{ {0,0,0,0,0,1,1,1},{2,2,0,0,3,3,3,0},{0,0,0,0,0,0,0,0}}
    \end{tikzpicture}\right)\\
    &=(u_{[4,7]}\cdot u_4u_7)\cdot (u_{[7,8]})\cdot \Lambda\left(\begin{tikzpicture}[scale=0.45, baseline={(0,0.25)}]
            \ppAff{5}{{0,1,1,1},{3,3,3,0},{0,0,0,0}}\end{tikzpicture}\, \right)  \\     
    &=(u_{[4,7]}\cdot u_4u_7)\cdot (u_{[7,8]})\cdot (u_{[2,4]}\cdot u_{[2,3]})\cdot \Lambda\left(\begin{tikzpicture}[scale=0.45, baseline={(0,0.25)}]
            \ppAff{4}{{1,1,1},{0,0,0},{0,0,0}}\end{tikzpicture}\, \right)  \\    
    &=(u_{[4,7]}\cdot u_4u_7)\cdot (u_{[7,8]})\cdot (u_{[2,4]}\cdot u_{[2,3]})\cdot u_{[1,3]}\\ &= 873126954.
\end{align*}
\end{example}
The next theorem is proved in \cite[Corollary 6.4]{RS17} and connects staircase diagrams with smooth permutations.
\begin{theorem}
\label{thm:smooth_bijection}
Let $\mcD$ be a staircase diagram on $[n-1]$.  Then the permutation $\Lambda(\mcD)$ is well-defined.
Furthermore, the map $\Lambda$ is a bijection:
$$\{\text{Staircase diagrams on $[n-1]$}\}\xrightarrow{\Lambda} \{w\in \mfS_n\ |\ w\text{ is smooth}\}.$$
\end{theorem}

For a staircase diagram $\mcD=\{B_1,\ldots, B_k\}$, define the \textbf{dual staircase diagram} $$\mcD^*:=\{B^*_1,\ldots, B^*_k\}$$ where $B^*_i=B_i$ set theoretically, and $B^*_i\prec B^*_j$ if and only if $B_i\succ B_j$ (i.e. the poset structure dual to $\mcD$).  It is clear from \Cref{def:staircase} that $\mcD^*$ is a well-defined staircase diagram.  For example, the dual to the staircase diagram in \Cref{Ex:typeA_11.1} has the picture:

$$\begin{tikzpicture}[scale=0.45]
            \ppAff{12}{{0,6,6,0,5,5,5,5,0,0,0},{4,0,0,2,2,0,0,3,3,3,0},{0,0,0,0,0,0,0,0,1,1,1}}
    \end{tikzpicture}$$
    
The next proposition, proved in \cite[Definition 3.5 and Theorem 3.8]{RS17}, connects several combinatorial properties of staircase diagrams with their associated permutations.  

\begin{proposition}\label{prop:staircase_permutation_properties}
Let $\mcD=\{B_1,\ldots, B_k\}$ be a staircase diagram on $[n-1]$ and let $\Lambda(\mcD)=w(1)\cdots w(n)$.  Then the following are true:
\begin{enumerate}
    \item The support $\displaystyle S(\Lambda(\mcD))= S(\mcD):=\bigcup_{i=1}^k B_i$.
    \item If $B_i\prec B_{i+1}$ and $B_i=[r_i,l_i]$, then $w(l_i+1)<w(l_i+2)$.
    \item If $B_i\succ B_{i+1}$ and $B_{i+1}=[r_{i+1},l_{i+1}]$, then $w(r_{i+1}-1)<w(r_{i+1})$.
    \item If $\mcD^*$ is the dual staircase diagram to $\mcD$, then $\Lambda(\mcD)^{-1}=\Lambda(\mcD^*)$.
\end{enumerate}    
\end{proposition}

Next we relate staircase diagrams to decomposing factorizations of permutations. 

\begin{proposition}\label{prop:staircase_wk_decomp}
    Let $k\geq 2$, $\mcD=\{B_1,\ldots, B_k\}$, and suppose $B_i\cap B_{i+1}=\varnothing$ for some $i<k$. Define the subdiagrams $$\mcD_1:=\{B_1,\ldots,B_i\}\quad\text{and}\quad \mcD_2:=\{B_{i+1},\ldots, B_k\}.$$  of $\mcD$. Then the following are true:
    \begin{enumerate}

        \item If $B_i\prec B_{i+1}$, then $\Lambda(\mcD)=\Lambda(\mcD_2)\cdot \Lambda(\mcD_1)$.
        \item If $B_i\succ B_{i+1}$, then $\Lambda(\mcD)=\Lambda(\mcD_1)\cdot \Lambda(\mcD_2)$.
        \item If $B_i\cup B_{i+1}$ is disconnected, then $\Lambda(\mcD)=\Lambda(\mcD_1)\cdot \Lambda(\mcD_2)=\Lambda(\mcD_2)\cdot \Lambda(\mcD_1)$.
    \end{enumerate}
Furthermore, the products in parts (1) and (2) are weakly decomposing factorizations and the product in part (3) is a decomposing factorization.
\end{proposition}

\begin{proof}
We prove the proposition by induction on $k=|\mcD|$. If $k=2$, then $\mcD=\{B_1,B_2\}$ with $B_1\cap B_{2}=\varnothing$. If $B_1\cup B_{2}$ is disconnected, then $B_1$ and $B_2$ are both maximal (and minimal) in $\mcD$.  Thus
$$\Lambda(\mcD)=u_{B_2}\cdot u_{J(B_2)}\cdot \Lambda(\{B_1\})=u_{B_1}\cdot u_{J(B_1)}\cdot \Lambda(\{B_2\}).$$
Since $B_1\cap B_{2}=\varnothing$, we have $u_{J(B_1)}=u_{J(B_2)}=\text{id}$ and hence
$$\Lambda(\mcD)=\Lambda(\{B_1\})\cdot \Lambda(\{B_2\})=\Lambda(\{B_2\})\cdot \Lambda(\{B_1\}).$$  
If $B_1\prec B_{2}$, then only $B_2$ is maximal in $\mcD$ and
$$\Lambda(\mcD)=u_{B_2}\cdot u_{J(B_2)}\cdot \Lambda(\{B_1\})=\Lambda(\{B_2\})\cdot \Lambda(\{B_1\}).$$
Similarly, if $B_1\succ B_{2}$, then only $B_1$ is maximal and $\Lambda(\mcD)=\Lambda(\{B_1\})\cdot \Lambda(\{B_2\})$.

Now suppose $\mcD=\{B_1,\ldots, B_k\}$ with $B_i\cap B_{i+1}=\varnothing$ for some $i<k$.  If $B_i\cup B_{i+1}$ is disconnected, then the blocks in $\mcD_1$ and $\mcD_2$ are pairwise incomparable in $\mcD$ and hence both sets contain a block that is maximal in $\mcD$.  Let $B_m$ be maximal in $\mcD$ such that $m\leq i$.  By induction, 
$$\Lambda(\mcD)=u_{B_m}\cdot u_{J(B_m)}\cdot \Lambda(\mcD\setminus\{B_m\})=u_{B_m}\cdot u_{J(B_m)}\cdot \Lambda(\mcD_1\setminus\{B_m\})\cdot\Lambda(\mcD_2).$$
Since $B_m\cap S(\mcD_2)=\varnothing$, the set $J(B_m)$ is unchanged when restricting from $\mcD$ to $\mcD_1$.  Thus  
$$u_{B_m}\cdot u_{J(B_m)}\cdot \Lambda(\mcD_1\setminus\{B_m\})=\Lambda(\mcD_1).$$
A similar argument shows that $\Lambda(\mcD)=\Lambda(\mcD_2)\cdot \Lambda(\mcD_1)$ by choosing a maximal block in $\mcD_2$.

If $B_i\prec B_{i+1}$, then there exists a maximal element $B_m\in \mcD$ such that $m>i$.  If $m>i+1$, then $B_i\prec B_{i+1}$ in $\mcD\setminus\{B_m\}$.
Induction on $k$ yields 
$$\Lambda(\mcD)=u_{B_m}\cdot u_{J(B_m)}\cdot \Lambda(\mcD\setminus\{B_m\})=u_{B_m}\cdot u_{J(B_m)}\cdot \Lambda(\mcD_2\setminus\{B_m\})\cdot\Lambda(\mcD_1)=\Lambda(\mcD_2)\cdot\Lambda(\mcD_1).$$
If $m=i+1$, then either $B_i\cup B_{i+2}$ is disconnected in $\mcD\setminus\{B_m\}$ or $\mcD_2=\{B_m\}$.  In either case, we still conclude $\Lambda(\mcD)=\Lambda(\mcD_2)\cdot\Lambda(\mcD_1).$  Note that if $B_i\succ B_{i+1}$, then an analogous argument shows that $\Lambda(\mcD)=\Lambda(\mcD_1)\cdot \Lambda(\mcD_2)$.

Finally, observe that if $B_i=[l_i,r_i]$ and $B_{i+1}=[l_{i+1},r_{i+1}]$, then $$S(\Lambda(\mcD_1))\subseteq [1,r_i]\quad \text{and}\quad S(\Lambda(\mcD_2))\subseteq [l_{i+1},n-1].$$  Since $B_i\cap B_{i+1}=\varnothing$, we have $r_i<l_{i+1}$.  \Cref{lem:wk_decomposable_support} implies the product  $\Lambda(\mcD_i)\cdot \Lambda(\mcD_{3-i})$ is a weakly decomposing factorization for $i=1,2$.  Furthermore, if $B_i\cup B_{i+1}$ is disconnected, then \Cref{lem:decomposable} implies the product is a decomposing factorization.  This completes the proof.
\end{proof}

\begin{example} The following is a weakly decomposing factorization with $[2,4]\prec [5,7]$. 

\begin{align*}
\Lambda\left(\begin{tikzpicture}[scale=0.45, baseline={(0,0.5)}]
            \ppAff{9}{ {0,0,0,0,0,1,1,1},{2,2,0,0,3,3,3,0},{0,5,5,5,0,0,0,0}}
    \end{tikzpicture}\right)&=\Lambda\left(\ \begin{tikzpicture}[scale=0.45, baseline={(0,0.5)}]
            \ppAff{9}{ {0,0,0,0,0,0,0,0},{2,2,0,0,0,0,0,0},{0,5,5,5,0,0,0,0}}
    \end{tikzpicture}\right)\cdot \Lambda\left(\begin{tikzpicture}[scale=0.45, baseline={(0,0.5)}]
            \ppAff{9}{ {0,0,0,0,0,1,1,1},{0,0,0,0,3,3,3,0},{0,0,0,0,0,0,0,0}}
    \end{tikzpicture}\right)\\
    &= (u_{[5,7]}\cdot u_{7}\cdot u_{[7,8]})\cdot (u_{[2,4]}\cdot u_{[2,3]}\cdot u_{[1,3]}).
\end{align*}

\end{example}

\Cref{lem:wk_decomposable_patterns} and \Cref{prop:staircase_wk_decomp} imply that, for pattern avoidance on smooth permutations, it suffices to study staircase diagrams $\mcD$ where $\Lambda(\mcD)$ is strongly indecomposable given that the avoiding patterns are also strongly indecomposable.  We end this section with a ``converse" of \Cref{prop:staircase_wk_decomp}.  

\begin{proposition}\label{prop:staircase_wk_decomp2}
Let $\mcD=\{B_1,\ldots, B_k\}$ be a staircase diagram on $[n-1]$ such that $\Lambda(\mcD)$ is weakly decomposable.  Then there exists $i<k$ such that $B_i\cap B_{i+1}=\varnothing$.  

Moreover, if $\Lambda(\mcD)$ is decomposable, then there exists $i<k$ such that $B_i\cup B_{i+1}$ is disconnected.
\end{proposition}

\begin{proof}
Suppose that $\Lambda(\mcD)$ is left-weakly decomposable with a decomposing factorization $\Lambda(\mcD)=w_1\cdot w_2$.  Since the smooth patterns 3412 and 4231 are both strongly indecompsoable, it follows from \Cref{lem:wk_decomposable_patterns} that $w_1$ and $w_2$ are both smooth permutations since $\Lambda(\mcD)$ is smooth. \Cref{thm:smooth_bijection} implies there exist staircase diagrams $\mcD_1$ on $[j-1]$ and $\mcD_2$ on $[j,n-1],$ where $w_1=\Lambda(\mcD_1)$ and $w_2=\Lambda(\mcD_2)$.  Let 
\[\mcD_1=\{B'_1,\ldots, B'_i\}\quad\text{and}\quad\mcD_2=\{B'_{i+1},\ldots, B'_{\ell}\}\]
and define 
\[\mcD'=\{B'_1,\ldots,B'_i,B'_{i+1},\ldots,B'_{\ell}\}\] with the staircase poset structures of $\mcD_1 $ and $\mcD_2$, and with the additional relation $B'_i\succ B'_{i+1}$ if $B'_i\cup B'_{i+1}$ is connected.  Since $B'_i\cap B'_{i+1}=\varnothing$, we have
\[\Lambda(\mcD)=w_1\cdot w_2=\Lambda(\mcD_1)\cdot\Lambda(\mcD_2)=\Lambda(\mcD').\]
But \Cref{thm:smooth_bijection} implies that $\Lambda(\mcD)$ uniquely determines $\mcD$ and hence $\mcD=\mcD'$ and $B_i\cap B_{i+1}=\varnothing$.  If $\Lambda(\mcD)$ is right-weakly decomposable, a similar argument applies with the same conclusion.  If $\Lambda(\mcD)$ is decomposable, again the argument gives the same conclusion, but with $B_i\cup B_{i+1}$ disconnected.
\end{proof}


\section{Staircase diagrams and pattern avoidance}\label{sec:sd_patterns}

In this section, we study staircase diagrams in relation to pattern avoidance.  We will often view $\Lambda(\mcD)$ as a permutation pattern.  In particular, if a staircase diagram has support $S(\mcD)=[a,b]$, we can translate the support to $[1,b-a+1]$ and hence view $\Lambda(\mcD)$ as a permutation in $\mfS_{b-a+2}$.  We use the convention that if $\Lambda(\mcD)$ is to be avoided or contained, then we view $\Lambda(\mcD)$ as a permutation pattern in $\mfS_{b-a+2}$.

The next two lemmas establish a containment principle on staircase diagrams and their corresponding permutations. 


\begin{lemma}\label{lem:containment_expand}
    Let $\mcD=\{B_1,\ldots, B_k\}$ be a staircase diagram on $[n-1]$ with $B_k=[m,n-1]$.
    With the induced poset structure from $\mcD$, define $$\mcD^+:=\{B_1,\ldots, B_{k-1},B_k'\}$$ where $B_k'=[m,n]$.  Then $\Lambda(\mcD^+)$ contains $\Lambda(\mcD)$.
\end{lemma}

 \begin{proof}
Since $B_k$ is the terminal block of $\mcD$, it is either maximal or minimal. Suppose that $B_k$ is maximal.  Since $\mcD\setminus\{B_k\}=\mcD^+\setminus\{B'_k\}$ and $J(B_k)=J(B'_k)$, we have 
 $$\Lambda(\mcD^+)=u_{B'_k}\cdot u_{J(B'_k)}\cdot \Lambda(\mcD^+\setminus\{B_k\})=u_{B'_k}\cdot u_{(B_k)}\cdot \Lambda(\mcD).$$
Observe that $u_{B'_k}\cdot u_{B_k}=s_m\cdots s_n$ is a cyclic permutation whose left action appends $n+1$ to $\Lambda(\mcD)$ and then shifts the values $(m,m+1,\ldots, n,n+1)$.  Specifically, if $\Lambda(\mcD)=w(1)\cdots w(n)$ in one-line notation, then $\Lambda(\mcD^+)=w'(1)\cdots w'(n+1)$ where 
$$w'(i)=\begin{cases}
          w(i) &\text{if  $w(i)<m$}\\
          w(i)+1 &\text{if  $m\leq w(i)\leq n$}\\
          m &\text{if  $i=n+1$.}
      \end{cases}$$
Hence $\Lambda(\mcD^+)$ contains $\Lambda(\mcD)$.
If $B_k$ is minimal in $\mcD$, then it is maximal in the dual staircase diagram $\mcD^*$.   \Cref{prop:staircase_permutation_properties} part (4) implies that
 $$\Lambda(\mcD^+)^{-1}=s_m\cdots s_n\cdot \Lambda(\mcD)^{-1}$$ and hence
  $$\Lambda(\mcD^+)=\Lambda(\mcD)\cdot s_n\cdots s_m.$$
  The right action of the cyclic shift $s_n\cdots s_m$ ``inserts" $n+1$ into the $m$-th position of $\Lambda(\mcD)$.  In other words,
  $$\Lambda(\mcD^+)=w(1)\cdots w(m-1)\ (n+1)\ w(m+1)\cdots  w(n)$$
  where $\Lambda(\mcD):=w(1)\cdots w(n)$.  Again we have $\Lambda(\mcD^+)$ contains $\Lambda(\mcD)$.
  \end{proof}

\begin{lemma}\label{lem:containment_delete}
    Let $\mcD=\{B_1,\ldots, B_k\}$ be a staircase diagram on $[n]$ and consider the subdiagram $\mcD'=\mcD\setminus\{B_k\}$.  Then $\Lambda(\mcD)$ contains $\Lambda(\mcD')$. 
\end{lemma}

\begin{proof}
By \Cref{lem:containment_expand}, we can assume, without loss of generality, that $S(\mcD)=[n]$ and $S(\mcD')=[n-1]$.  If $B_k$ is maximal in $\mcD$, then 
$$\Lambda(\mcD)=u_{B_k}\cdot u_{J(B_k)}\cdot \Lambda(\mcD').$$

If $B_k=[m,n]$, then $J(B_k)=[m,n-1]$ and $u_{B_k}\cdot u_{J(B_k)}=s_m\cdots s_n$ is a cyclic permutation whose left action on $\Lambda(\mcD')$ appends $n+1$ and then shifts the values $(m,m+1,\dots, n+1)$.  We conclude that $\Lambda(\mcD)$ contains $\Lambda(\mcD')$.

If  $B_k$ is minimal in $\mcD$, then we can again conclude that $\Lambda(\mcD)$ contains $\Lambda(\mcD')$ by considering the dual staircase diagram $\mcD^*$ as in the proof of Lemma \ref{lem:containment_expand}.
\end{proof}

\begin{example}\label{ex:staircase_containment}
The following are examples of $\mcD'$, $\mcD$ and $\mcD^+$ respectively as in \Cref{lem:containment_expand} and \Cref{lem:containment_delete}.    
    $$
    \begin{tikzpicture}[scale=0.44]
        \ppAff{8}{{0,3,3,3,3},{0,0,0, 0,1,1,1}}
    \end{tikzpicture}    
    \qquad 
    \begin{tikzpicture}[scale=0.44]
            \ppAff{8}{{0,3,3,3,3},{4, 4,0, 0,1,1,1}}
    \end{tikzpicture}    
    \qquad     
    \begin{tikzpicture}[scale=0.44]
            \ppAff{9}{{0,0,3,3,3,3},{4,4, 4,0, 0,1,1,1}}
    \end{tikzpicture}
    $$
We have that 
\[\Lambda(\mcD')=4376521,\quad \Lambda(\mcD)=43875216,\quad \Lambda(\mcD^+) = 439852176.\]
$$
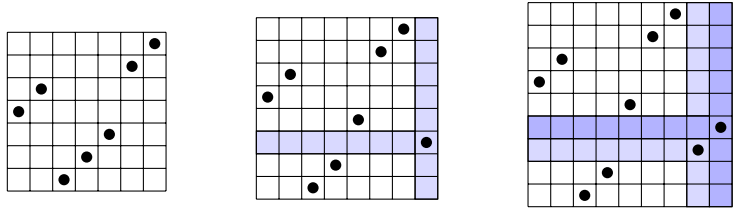
\begin{tikzpicture}[scale=0.3]
\draw[step=1.0,black] (0,0) grid (7,7);
\fill (0.5,3.5) circle (7pt);
\fill (1.5,4.5) circle (7pt);
\fill (2.5,0.5) circle (7pt);
\fill (3.5,1.5) circle (7pt);
\fill (4.5,2.5) circle (7pt);
\fill (5.5,5.5) circle (7pt);
\fill (6.5,6.5) circle (7pt);
\end{tikzpicture}\qquad \quad
\begin{tikzpicture}[scale=0.3,baseline={(0,0.1)}]
\draw[fill=blue!15] (0,2) rectangle (7,3);
\draw[fill=blue!15] (7,0) rectangle (8,8);
\draw[step=1.0,black] (0,0) grid (8,8);
\fill (0.5,4.5) circle (7pt);
\fill (1.5,5.5) circle (7pt);
\fill (2.5,0.5) circle (7pt);
\fill (3.5,1.5) circle (7pt);
\fill (4.5,3.5) circle (7pt);
\fill (5.5,6.5) circle (7pt);
\fill (6.5,7.5) circle (7pt);
\fill (7.5,2.5) circle (7pt);
\end{tikzpicture}\qquad \quad
\begin{tikzpicture}[scale=0.3,baseline={(0,0.2)}]
\draw[fill=blue!15] (0,2) rectangle (7,3);
\draw[fill=blue!15] (7,4) rectangle (8,9);
\draw[fill=blue!15] (7,0) rectangle (8,3);
\draw[fill=blue!30] (8,0) rectangle (9,9);
\draw[fill=blue!30] (0,3) rectangle (8,4);
\draw[step=1.0,black] (0,0) grid (9,9);
\fill (0.5,5.5) circle (7pt);
\fill (1.5,6.5) circle (7pt);
\fill (2.5,0.5) circle (7pt);
\fill (3.5,1.5) circle (7pt);
\fill (4.5,4.5) circle (7pt);
\fill (5.5,7.5) circle (7pt);
\fill (6.5,8.5) circle (7pt);
\fill (7.5,2.5) circle (7pt);
\fill (8.5,3.5) circle (7pt);
\end{tikzpicture}
$$
\end{example}



We see in \Cref{ex:staircase_containment} that $\Lambda(\mcD')$ is contained in $\Lambda(\mcD),$ which is contained in $\Lambda(\mcD^+)$.  The nature of the containment from \Cref{lem:containment_expand} and \Cref{lem:containment_delete} depends on whether $B_k$ is a maximal or minimal block in $\mcD$.  In \Cref{ex:staircase_containment}, we showcase when $B_k$ is maximal in $\mcD$.  Note that the containment manifests itself as ``appending" on the right side of the matrix and bumping large numbers accordingly.  When $B_k$ is a minimal block, we append to the bottom of the matrix and shift to the right (i.e. the transpose of the diagrams in \Cref{ex:staircase_containment}).  The action of appending/removing a block at the end of a staircase diagram will be an important operation in the following sections, so we summarize the process.  

Let $\mcD=\{B_1,\ldots, B_k\}$ be a staircase diagram on $[n-1]$ such that $B_{k}=[r,n-1]$ and $B_{k-1}\cap B_k=[r,m-1]$ and let $\mcD'=\mcD\setminus\{B_k\}$ denote the sub-staircase diagram on $[m-1]$.  Let $\Lambda(\mcD)=w(1)\cdots w(n)$ and $\Lambda(\mcD')=w'(1)\cdots w'(r)$.  If $B_k$ is maximal in $\mcD$, then \Cref{lem:containment_expand} and \Cref{lem:containment_delete} imply
\begin{equation}\label{eqn:B_k_max_block}
w(i)=\begin{cases}
    w'(i) & \text{if $w'(i)< r$}\\
    w'(i)+n-r & \text{if $r\leq w'(i)\leq m$}\\
    n+m-i & \text{if $m< i\leq n$}\\
\end{cases},
\end{equation}
and if $B_k$ is minimal, then 
\begin{equation}\label{eqn:B_k_min_block} 
w(i)=\begin{cases}
    w'(i) & \text{if $i<r$}\\
    n+m-i & \text{if $r\leq i\leq r+n-m$}\\
    w'(i-n+m) & \text{if $i>r+n-m$}
\end{cases}.
\end{equation}
This relationship between $\Lambda(\mcD)$ and $\Lambda(\mcD')$ can be visualized by dividing the permutation matrix of $\Lambda(\mcD')$ into two regions corresponding to the overlap $B_{k-1}\cap B_k$ as in \Cref{fig:matrix_of_sub_D}.  The regions are oriented either horizontally or vertically depending on if $B_k$ is either maximal (regions $A_1, A_2$) or minimal (regions $A_3, A_4$).


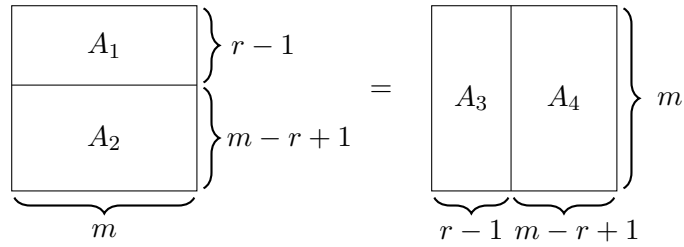
\begin{figure}[h!]
$$\begin{tikzpicture}[scale=0.35,baseline={(0,1.25)}]
\draw (0,0) rectangle (7,7);
\draw (0,4) -- (7,4);
\draw [decorate, decoration={brace, mirror, amplitude=6pt},thick] (7.25,4.1) -- (7.25,6.9);
\draw (9.5,5.5)node {$r-1$};
\draw [decorate, decoration={brace, mirror, amplitude=6pt},thick] (7.25,0.1) -- (7.25,3.9);
\draw (10.5,2)node {$m-r+1$};
\draw [decorate, decoration={brace, mirror, amplitude=6pt},thick] (0.1,-0.25) -- (6.9,-0.25);
\draw (3.5,-1.5)node {$m$};
\draw (3.5,5.5)node {$A_1$};
\draw (3.5,2)node {$A_2$};
\end{tikzpicture} = \quad
\begin{tikzpicture}[scale=0.35,baseline={(0,1.25)}]
\draw (0,0) rectangle (7,7);
\draw (3,0) -- (3,7);
\draw [decorate, decoration={brace, mirror, amplitude=6pt},thick] (7.25,0.1) -- (7.25,6.9);
\draw (9,3.5)node {$m$};
\draw [decorate, decoration={brace, mirror, amplitude=6pt},thick] (0.1,-0.25) -- (2.9,-0.25);
\draw (1.5,-1.5)node {$r-1$};
\draw [decorate, decoration={brace, mirror, amplitude=6pt},thick] (3.1,-0.25) -- (6.9,-0.25);
\draw (5.5,-1.5)node {$m-r+1$};
\draw (1.5,3.5)node {$A_3$};
\draw (5,3.5)node {$A_4$};
\end{tikzpicture}$$
\caption{The matrix of $\Lambda(\mcD')$ split along $B_{k-1}\cap B_{k}=[r,m-1]$.}\label{fig:matrix_of_sub_D}
\end{figure}

The matrix of $\Lambda(\mcD)$ is obtained by inserting an anti-diagonal square of dimension $(n-m)$ either on the right side (if $B_k$ is maximal) or bottom side (if $B_k$ is minimal) of $\Lambda(\mcD')$ between the two regions as in \Cref{fig:matrix_append_B_k}.

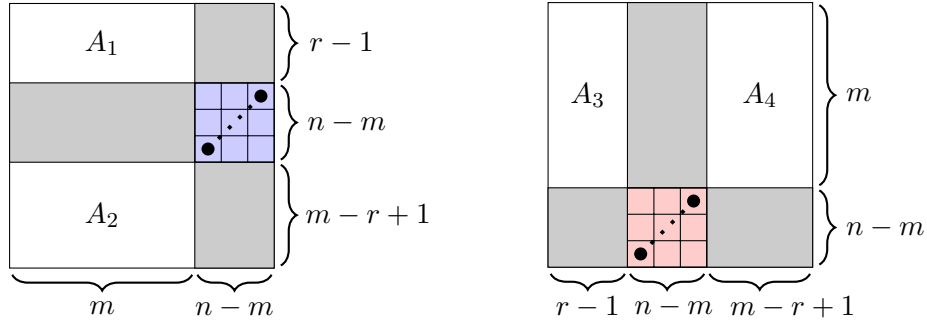
\begin{figure}[h!]
\begin{tikzpicture}[scale=0.35]
\draw (0,0) rectangle (10,10);
\draw[fill=blue!20] (7,4) rectangle (10,7);
\draw[fill=black!20] (0,4) rectangle (7,7);
\draw[fill=black!20] (7,0) rectangle (10,4);
\draw[fill=black!20] (7,7) rectangle (10,10);
\draw[step=1.0,black] (7,4) grid (10,7);
\fill (7.5,4.5) circle (7pt);
\fill (9.5,6.5) circle (7pt);
\draw[loosely dotted,thick,line width = 0.5mm] (7.5,4.5) -- (9.5,6.5);
\draw [decorate, decoration={brace, mirror, amplitude=6pt},thick] (10.25,4.1) -- (10.25,6.9);
\draw (12.75, 5.5) node {$n-m$};
\draw [decorate, decoration={brace, mirror, amplitude=6pt},thick] (10.25,7.1) -- (10.25,9.9);
\draw (12.55, 8.5)node {$r-1$};
\draw [decorate, decoration={brace, mirror, amplitude=6pt},thick] (10.25,0.1) -- (10.25,3.9);
\draw (13.55,2)node {$m-r+1$};
\draw [decorate, decoration={brace, mirror, amplitude=6pt},thick] (7.1,-0.25) -- (9.9,-0.25);
\draw (8.5,-1.5)node {$n-m$};
\draw [decorate, decoration={brace, mirror, amplitude=6pt},thick] (0.1,-0.25) -- (6.9,-0.25);
\draw (3.5,-1.5)node {$m$};
\draw (3.5,8.5)node {$A_1$};
\draw (3.5,2)node {$A_2$};
\end{tikzpicture}\hspace{0.5in} \begin{tikzpicture}[scale=0.35]
\draw (0,0) rectangle (10,10);
\draw[fill=red!20] (3,0) rectangle (6,3);
\draw[fill=black!20] (3,3) rectangle (6,10);
\draw[fill=black!20] (0,0) rectangle (3,3);
\draw[fill=black!20] (6,0) rectangle (10,3);
\draw[step=1.0,black] (3,0) grid (6,3);
\fill (3.5,0.5) circle (7pt);
\fill (5.5,2.5) circle (7pt);
\draw[loosely dotted,thick,line width = 0.5mm] (3.5,0.5) -- (5.5,2.5);
\draw [decorate, decoration={brace, mirror, amplitude=6pt},thick] (10.25,3.1) -- (10.25,9.9);
\draw (11.75,6.5)node {$m$};
\draw [decorate, decoration={brace, mirror, amplitude=6pt},thick] (10.25,0.1) -- (10.25,2.9);
\draw (12.75,1.5)node {$n-m$};
\draw [decorate, decoration={brace, mirror, amplitude=6pt},thick] (0.1,-0.25) -- (2.9,-0.25);
\draw (1.5,-1.5)node {$r-1$};
\draw [decorate, decoration={brace, mirror, amplitude=6pt},thick] (3.1,-0.25) -- (5.9,-0.25);
\draw (4.75,-1.5)node {$n-m$};
\draw [decorate, decoration={brace, mirror, amplitude=6pt},thick] (6.1,-0.25) -- (9.9,-0.25);
\draw (9.25,-1.5)node {$m-r+1$};
\draw (1.5,6.5)node {$A_3$};
\draw (8,6.5)node {$A_4$};
\end{tikzpicture}
\caption{The matrix of $\Lambda(\mcD)$ relative to $\Lambda(\mcD')$ when $B_k$ is respectively maximal and minimal.}\label{fig:matrix_append_B_k}
\end{figure}

We say a region $A$ of permutation (or permutation matrix) is \textbf{decreasing} if for every $i<j$ such that $(i,w(i)),(j,w(j))\in A$, we have $w(i)>w(j)$. 

\begin{lemma}\label{lem:regions_A_2_4_decreasing}
Let $\mcD=\{B_1,\ldots,B_k\}$ be a staircase diagram.  If $B_k$ is maximal, then region $A_2$ from \Cref{fig:matrix_append_B_k} is decreasing in $\Lambda(\mcD)$.  Similarly, if $B_k$ is minimal, then region $A_4$ is decreasing in $\Lambda(\mcD)$.
\end{lemma}

\begin{proof}
By the duality from \Cref{prop:staircase_permutation_properties}, part (4), it suffices to assume that $B_k$ is maximal and show that region $A_2$ from \Cref{fig:matrix_append_B_k} is decreasing.  Suppose that $A_2$ is not decreasing and therefore contains nodes $(i,w(i))$ and $(j,w(j))$ such that $i<j$ and $w(i)<w(j)$.  Since $B_{k-1}\prec B_k$, \Cref{prop:staircase_permutation_properties}, part (1) implies that $w(m)<w(m+1)$.  But now $\Lambda(\mcD)$ contains the pattern 3412, which contradicts the fact that it is a smooth permutation (see \Cref{fig:perm_contains_3412}).  Thus region $A_2$ must be decreasing.
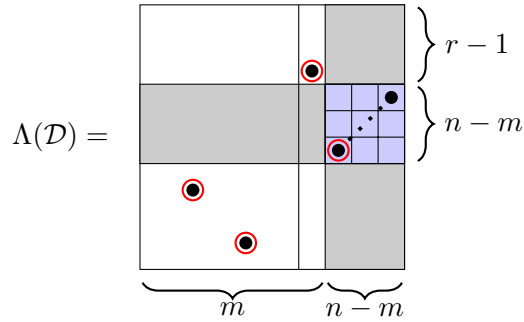
\begin{figure}[h!]
\begin{tikzpicture}[scale=0.35]
\draw (0,0) rectangle (10,10);
\draw[fill=black!20] (7,7) rectangle (10,10);
\draw[fill=black!20] (0,4) rectangle (7,7);
\draw[fill=black!20] (7,0) rectangle (10,4);
\draw[fill=blue!20] (7,4) rectangle (10,7);
\draw[step=1.0,black] (7,4) grid (10,7);
\draw (6,0) -- (6,10);
\fill (7.5,4.5) circle (7pt);
\fill (9.5,6.5) circle (7pt);
\draw[loosely dotted,thick,line width = 0.5mm] (7.5,4.5) -- (9.5,6.5);
\draw[thick, red]  (7.5,4.5) circle (11pt);
\draw [decorate, decoration={brace, mirror, amplitude=6pt},thick] (10.5,7.1) -- (10.5,9.9);
\draw (12.75,8.5)node {$r-1$};
\draw [decorate, decoration={brace, mirror, amplitude=6pt},thick] (10.5,4.1) -- (10.5,6.9);
\draw (13,5.5)node {$n-m$};
\draw [decorate, decoration={brace, mirror, amplitude=6pt},thick] (0.1,-0.5) -- (6.9,-0.5);
\draw (3.5,-1.5)node {$m$};
\draw [decorate, decoration={brace, mirror, amplitude=6pt},thick] (7.1,-0.5) -- (9.9,-0.5);
\draw (8.5,-1.5)node {$n-m$};
\draw (-3,5)node {$\Lambda(\mcD)=$};
\fill (2,3) circle (7pt);
\fill (4,1) circle (7pt);
\draw[thick, red]  (2,3) circle (11pt);
\draw[thick, red]  (4,1) circle (11pt);
\fill (6.5,7.5) circle (7pt);
\draw[thick, red]  (6.5,7.5) circle (11pt);
\end{tikzpicture}
\caption{The permutation $\Lambda(\mcD)$ containing 3412 if region $A_2$ is not decreasing.}\label{fig:perm_contains_3412}
\end{figure}
\end{proof}

Note that there are symmetrical analogues of \Cref{lem:containment_expand},  \Cref{lem:containment_delete} and \Cref{lem:regions_A_2_4_decreasing} where we replace the last block $B_k$ with the first block $B_1$ in $\mcD$.  This symmetry can be seen by applying the group automorphism that sends $s_i\mapsto s_{n-i}$ to the permutation $\Lambda(\mcD)$.  The following proposition is a consequence of \Cref{lem:containment_expand} and \Cref{lem:containment_delete} and their symmetrical ``$B_1$" analogues.

\begin{proposition}\label{prop:subdiagram_containment}
    Let $\mcD=\{B_1,\ldots, B_k\}$ be a staircase diagram and let $\mcD'=\{B_x,\ldots, B_y\}$ denote a subdiagram of consecutive blocks in $\mcD$.  Then $\Lambda(\mcD)$ contains $\Lambda(\mcD')$.
\end{proposition}

In the next two subsections, we consider two families staircase diagrams that are related to \Cref{thm:semi-polished_GF12}.


\subsection{Overlapping blocks}
Given two consecutive blocks $B_i$ and $B_{i+1}$ in a staircase diagram, we call the size of their intersection the \textbf{overlap} between $B_i$ and $B_{i+1}$.  In this section, we study how overlaps between adjacent blocks in a staircase diagram relate to permutation pattern avoidance.  

\begin{definition}
 Let $\mcD=\{B_1,\ldots,B_k\}$ be a staircase diagram.  We say $\mcD$ \textbf{contains a $q$-overlap} if there exists blocks $B_i$ and $B_{i+1}$ such that $|B_i\cap B_{i+1}|\geq q$.  Otherwise, we say $\mcD$\textbf{ avoids $q$-overlaps}.   
\end{definition}

We consider two staircase diagrams of minimal support size that contain a $q$-overlap.  Define 
$$\mcOv^+_q:=\{[1,q+1]\prec [2,q+2]\}\quad\text{and}\quad \mcOv^-_q:=\{[1,q+1]\succ [2,q+2]\}.$$
We will remove the numbers from pictures of staircase diagrams if they are not relevant.  Pictorially, we have  
$$
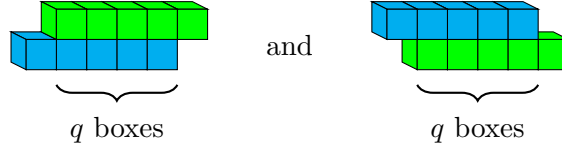
\begin{tikzpicture}[scale=0.4,baseline={(0,0.2)}]
            \ppsansnumbers{6}{{0,1,1,1,1,1},{2,2,2,2,2,0}}
            \draw [decorate, decoration={brace, mirror, amplitude=6pt},thick] (-4,-0.5) -- (0,-0.5);
            \draw (-2,-2)node {$q$ boxes};\end{tikzpicture}\qquad\text{and}\qquad
            \begin{tikzpicture}[scale=0.4,baseline={(0,0.2)}]
            \ppsansnumbers{6}{{2,2,2,2,2,0},{0,1,1,1,1,1}}
            \draw [decorate, decoration={brace, mirror, amplitude=6pt},thick] (-4,-0.5) -- (0,-0.5);
            \draw (-2,-2)node {$q$ boxes};\end{tikzpicture}
            $$
and since the support $S(\mcOv^{\pm}_q)=[q+2]$, the permutations $\Lambda(\mcOv^{\pm}_q)\in \mfS_{q+3}$.  

\begin{lemma}\label{lem:overlaps_1}
If $q\geq 0$, then 
$$\Lambda(\mcOv^+_q)=(q+3)(q+2)\cdots 4312\quad\text{and}\quad\Lambda(\mcOv^-_q)={(q+2)(q+3)(q+1)q}\cdots 321.$$

$$
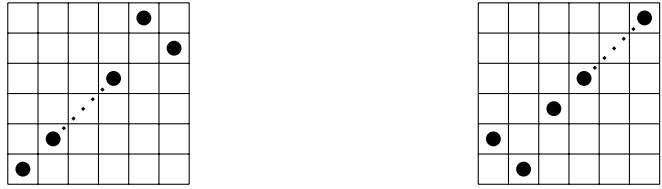
\begin{tikzpicture}[scale=0.4]
\draw[step=1.0,black] (0,0) grid (6,6);
\fill (0.5,0.5) circle (7pt);
\fill (1.5,1.5) circle (7pt);
\fill (3.5,3.5) circle (7pt);
\fill (4.5,5.5) circle (7pt);
\fill (5.5,4.5) circle (7pt);
\draw[loosely dotted,thick,line width = 0.4mm] (1.5,1.5) -- (3.5,3.5);
\end{tikzpicture}\hspace{1.5in}\begin{tikzpicture}[scale=0.4]
\draw[step=1.0,black] (0,0) grid (6,6);
\fill (0.5,1.5) circle (7pt);
\fill (1.5,0.5) circle (7pt);
\fill (2.5,2.5) circle (7pt);
\fill (3.5,3.5) circle (7pt);
\fill (5.5,5.5) circle (7pt);
\draw[loosely dotted,thick,line width = 0.4mm] (3.5,3.5) -- (5.5,5.5);
\end{tikzpicture}$$
Moreover, if a staircase diagram $\mcD$ contains a $q$-overlap, then $\Lambda(\mcD)$ contains at least one of $\Lambda(\mcOv^{+}_q)$ or $\Lambda(\mcOv^{-}_q).$
\end{lemma}

\begin{proof}
The one-line notation for $\Lambda(\mcOv^{+}_q)$ or $\Lambda(\mcOv^{-}_q)$ follows directly from \Cref{fig:matrix_of_sub_D} and \Cref{fig:matrix_append_B_k}.  For the second part, suppose that $\mcD$ contains a $q$-overlap.  Then there exist blocks $B_i,B_{i+1}\in \mcD$ such that $|B_i\cap B_{i+1}|=q'\geq q$.  \Cref{prop:subdiagram_containment} implies that $\Lambda(\mcD)$ contains one of  $\Lambda(\mcOv^{+}_{q'})$ or $\Lambda(\mcOv^{-}_{q'})$, which contain $\Lambda(\mcOv^{+}_q)$ or $\Lambda(\mcOv^{-}_q)$, respectively.
\end{proof}

Observe that the southwest (resp. northeast) anti-diagonal block of $\Lambda(\mcOv^{+}_q)$ (resp. $\Lambda(\mcOv^{-}_q)$ is of size $q+1$. Next we state our main theorem on $q$-overlaps.

\begin{theorem}\label{thm:overlaps}
Let $q\geq 1$.  The following are equivalent:
\begin{enumerate}
    \item The staircase diagram $\mcD$ avoids $q$-overlaps.
    \item The permutation $\Lambda(\mcD)$ avoids the patterns $\Lambda(\mcOv^{\pm}_q)$.
\end{enumerate}
\end{theorem}

\begin{proof}
If $\mcD$ contains a $q$-overlap, then \Cref{lem:overlaps_1} implies $\Lambda(\mcD)$ contains one of $\Lambda(\mcOv^{+}_q)$ or $\Lambda(\mcOv^{-}_q)$.  

For the converse, we assume that $\mcD=\{B_1,\ldots, B_k\}$ avoids $q$-overlaps and proceed by induction on $k$.  First, if $k=1$, then $\mcD$ vacuously avoids $q$-overlaps and $\Lambda(\mcD)=u_{B_1}$ clearly avoids $\Lambda(\mcOv^{\pm}_q)$.  Now suppose $k>1$ and let $\mcD':=\{B_1,\ldots,B_{k-1}\}$ denote the sub-diagram of $\mcD$ removing $B_k$.  Note that $\mcD'$ also avoids $q$-overlaps, so by induction, the permutation $\Lambda(\mcD')$ avoids $\Lambda(\mcOv^{\pm}_q)$.  Suppose $B_k$ is maximal in $\mcD$ and let $q':=|B_{k-1}\cap B_k|$ denote the overlap between $B_{k-1}$ and $B_k$.  In \Cref{fig:Overlaps_1}, we display the matrices of $\Lambda(\mcD')$ and $\Lambda(\mcD)$ following \Cref{fig:matrix_of_sub_D} noting that region $A_2$ has $q'+1$ rows. Since $\Lambda(\mcD')$ avoids $\Lambda(\mcOv^{+}_q)$, any containment of $\Lambda(\mcOv^{+}_q)$ in $\Lambda(\mcD)$ must involve region $A_3$.  Furthermore, region $A_3$ is decreasing and hence, any containment of $\Lambda(\mcOv^{+}_q)$ in $\Lambda(\mcD)$ must also use $q+1$ entries from region $A_2$.  If $\Lambda(\mcD)$ contains $\Lambda(\mcOv^{+}_q)$, then $q\leq q'$.  But $\mcD$ avoids $q$-overlaps by assumption and therefore $q'<q$.  This implies $\Lambda(\mcD)$ avoids $\Lambda(\mcOv^{+}_q)$.  

\begin{figure}[h!]
\begin{tikzpicture}[scale=0.35,baseline={(0,-0.75)}]
\draw (-3,3)node {$\Lambda(\mcD')=$};
\draw (0,0) rectangle (7,7);
\draw (0,4) -- (7,4);
\draw [decorate, decoration={brace, mirror, amplitude=6pt},thick] (7.5,0.1) -- (7.5,3.9);
\draw (9.75,2.25)node {$q'+1$};
\draw (3.5,5.5)node {$A_1$};
\draw (3.5,2)node {$A_2$};
\end{tikzpicture}\ \quad \begin{tikzpicture}[scale=0.4]
\draw (0,0) rectangle (10,10);
\draw[fill=black!20] (7,7) rectangle (10,10);
\draw[fill=black!20] (0,4) rectangle (7,7);
\draw[fill=black!20] (7,0) rectangle (10,4);
\draw[fill=blue!20] (7,4) rectangle (10,7);
\draw [decorate, decoration={brace, mirror, amplitude=6pt},thick] (10.5,0.1) -- (10.5,3.9);
\draw (12.75,2.25)node {$q'+1$};
\draw (6,0) -- (6,10);
\fill (9.5,6.5) circle (7pt);
\draw[thick, red]  (9.5,6.5) circle (11pt);
\fill (6.5,8.5) circle (7pt);
\draw[thick, red]  (6.5,8.5) circle (11pt);
\fill (2.5,3.5) circle (7pt);
\fill (0.5,1.5) circle (7pt);
\draw[loosely dotted,thick,line width = 0.4mm] (0.5,1.5) -- (2.5,3.5);
\draw[thick, red]  (0.5,1.5) circle (11pt);
\draw[thick, red]  (2.5,3.5) circle (11pt);
\draw (3.5,8.5)node {$A_1$};
\draw (3.5,2)node {$A_2$};
\draw (8.5,5.5)node {$A_3$};
\draw (-3,5)node {$\Lambda(\mcD)=$};
\end{tikzpicture}
\caption{If $\Lambda(\mcD)$ contains $\Lambda(\mcOv^{+}_q)$, then $q\leq q'$.}
\label{fig:Overlaps_1}
\end{figure}

Next we show that $\Lambda(\mcD)$ avoids $\Lambda(\mcOv^{-}_q)$.  Continuing with the assumption that $B_k$ is maximal in $\mcD$, we have that the region $A_2$ is decreasing by \Cref{lem:regions_A_2_4_decreasing} in \Cref{fig:Overlaps_1}.  Since both regions $A_2$ and $A_3$ are decreasing, any containment of $\Lambda(\mcOv^{-}_q)$ in $\Lambda(\mcD)$ must be within region $A_1$.  But $\Lambda(\mcD')$ avoids $\Lambda(\mcOv^{-}_q)$ and therefore, so does $\Lambda(\mcD)$.

Finally, if $B_k$ is minimal, we can apply the argument above to the dual staircase diagram $\mcD^*$ and use \Cref{prop:staircase_permutation_properties} part (4) noting that diagrams $\mcOv^{\pm}_q$ are dual to each other.  This completes the proof.
\end{proof}

\subsection{Strongly connected chains}
A staircase diagram is an \textbf{increasing chain of length $\ell$} if it has the form $$\{B_1\prec\cdots\prec B_{\ell}\}.$$  Similarly, we have \textbf{decreasing chains of length $\ell$} of the form $\{B_1\succ\cdots\succ B_{\ell}\}$.  

We call a staircase diagram $\mcD=\{B_1,\ldots, B_k\}$ \textbf{strongly connected} if $B_i\cap B_{i+1}\neq \varnothing$ for all $1\leq i< k$.  If the support $S(\mcD)=[a,b]$, then \Cref{prop:staircase_wk_decomp} and \Cref{prop:staircase_wk_decomp2} imply that $\mcD$ is strongly connected if and only if $\Lambda(\mcD)$ is strongly indecomposable in $\mfS_{b-a+2}$.

\begin{definition}
We say a staircase diagram $\mcD=\{B_1,\ldots,B_k\}$ \textbf{contains a strongly connected $\ell$-chain} if there exists a subdiagram $\{B_{m+1},\ldots,B_{m+\ell}\}$ that is a strongly connected chain of length $\ell$ (either increasing or decreasing).  Otherwise, we say $\mcD$ \textbf{avoids strongly connected $\ell$-chains.}
\end{definition}

There are two strongly connected $\ell$-chains of minimal support size given by consecutive blocks of the form $B_i=[i,i+1]$ arranged in either increasing or decreasing order.  We denote these staircase diagrams by $\mcCh^+_\ell$ and $\mcCh^-_\ell$ respectively.  

$$\mcCh^+_\ell = \begin{tikzpicture}[scale=0.4,baseline={(0,0.8)}]
            \ppsansnumbers{6}{{0,0,0,1,1},{0,0,2,2,0},{0,3,3,0,0},{4,4,0,0,0}}
            \draw [decorate, decoration={brace, mirror, amplitude=6pt},thick] (1.5,0) -- (1.5,4);
            \draw (4,2)node {$\ell$ blocks};
    \end{tikzpicture}\quad\text{and}\quad
\mcCh^-_\ell =\begin{tikzpicture}[scale=0.4,baseline={(0,0.8)}]
            \ppsansnumbers{6}{{4,4,0,0,0},{0,3,3,0,0},{0,0,2,2,0},{0,0,0,1,1}}
            \draw [decorate, decoration={brace, mirror, amplitude=6pt},thick] (1.5,0) -- (1.5,4);
            \draw (4,2)node {$\ell$ blocks};
    \end{tikzpicture}$$

Note that the diagrams $\mcCh^{\pm}_\ell$ have support $[\ell+1]$ and hence $\Lambda(\mcCh^{\pm}_\ell)\in \mfS_{\ell+2}$.  Specifically, we will show in \Cref{lem:chain_containment} below that
$$\Lambda(\mcCh^{+}_\ell)=(\ell+2)(\ell+1)12\cdots \ell\quad\text{and}\quad \Lambda(\mcCh^{-}_\ell)=34\cdots(\ell+1)(\ell+2) 21.$$
$$
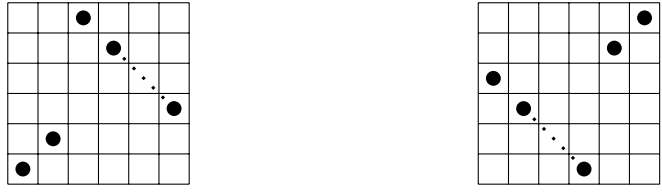
\begin{tikzpicture}[scale=0.4]
\draw[step=1.0,black] (0,0) grid (6,6);
\fill (0.5,0.5) circle (7pt);
\fill (1.5,1.5) circle (7pt);
\fill (2.5,5.5) circle (7pt);
\fill (3.5,4.5) circle (7pt);
\fill (5.5,2.5) circle (7pt);
\draw[loosely dotted,thick,line width = 0.4mm] (3.5,4.5) -- (5.5,2.5);
\end{tikzpicture}\hspace{1.5in}\begin{tikzpicture}[scale=0.4]
\draw[step=1.0,black] (0,0) grid (6,6);
\fill (0.5,3.5) circle (7pt);
\fill (1.5,2.5) circle (7pt);
\fill (3.5,0.5) circle (7pt);
\fill (4.5,4.5) circle (7pt);
\fill (5.5,5.5) circle (7pt);
\draw[loosely dotted,thick,line width = 0.4mm] (1.5,2.5) -- (3.5,0.5);
\end{tikzpicture}
$$

The goal of this subsection is to prove the following result relating strongly connected $\ell$-chains and permutation pattern avoidance.

\begin{theorem}\label{thm:strong_chains}
Let $\ell\geq 2$.  The following are equivalent:
\begin{enumerate}
    \item The staircase diagram $\mcD$ avoids strongly connected $\ell$-chains.
    \item The permutation $\Lambda(\mcD)$ avoids the patterns $\Lambda(\mcCh^{\pm}_\ell)$.
\end{enumerate}
\end{theorem}

\begin{lemma}\label{lem:chain_containment}
Let $\mcD=\{B_1,\ldots,B_k\}$ be a strongly connected, increasing chain on $[n-1]$ where $B_i=[r_i,l_i]$ for each $1\leq i\leq k$.  If $\Lambda(\mcD)=w(1)\cdots w(n)$, then the subsequence $$w(1)w(2)w(l_1+1)w(l_2+1)\cdots w(l_k+1)$$ is a containment of pattern $(k+2)(k+1)12\cdots k$ in $\Lambda(\mcD)$.  In particular, $$\Lambda(\mcCh_k^+)=(k+2)(k+1)12\cdots k.$$ 

Furthermore, by duality (i.e. \Cref{prop:staircase_permutation_properties} part (4)), if $\mcD=\{B_1,\ldots,B_k\}$ is a strongly connected, decreasing chain, then $\Lambda(\mcD)$ contains 
$$\Lambda(\mcCh_k^-)=\Lambda(\mcCh_k^+)^{-1}=34\cdots(k+1)(k+2) 21.$$
\end{lemma}

\begin{proof}
We prove the lemma by induction on the number of blocks $|\mcD|=k$.  First if $k=1$, then $\mcD=\{[1,n-1]\}$ and $\Lambda(\mcD)=u_{[1,n-1]}=n(n-1)\cdots 1$, which contains 321 using any subsequence of length 3.  In particular, the subsequence $w(1)w(2)w(n)=n(n-1)1$ is a containment of 321. 

Now suppose that $\mcD=\{B_1,\ldots,B_k\}$ is a strongly connected increasing chain on $[n-1]$ and let $\mcD'=\mcD\setminus\{B_k\}$ (in particular, $B_k$ is maximal). In \Cref{fig:strongchains_contain_1}, we show $\Lambda(\mcD')$ dividing into regions $A_1$ and $A_2$ in its containment in $\Lambda(\mcD)$.

\begin{figure}[h!]
\begin{tikzpicture}[scale=0.35,baseline={(0,-1.75)}]
\draw (-3,3)node {$\Lambda(\mcD')=$};
\draw (0,0) rectangle (7,7);
\draw (0,4) -- (7,4);
\draw [decorate, decoration={brace, mirror, amplitude=6pt},thick] (7.5,4.1) -- (7.5,6.9);
\draw (9.75,5.5)node {$r_k-1$};
\draw [decorate, decoration={brace, mirror, amplitude=6pt},thick] (0.1,-0.5) -- (6.9,-0.5);
\draw (3.5,-1.5)node {$l_{k-1}+1$};
\draw (3.5,5.5)node {$A_1$};
\draw (3.5,2)node {$A_2$};
\end{tikzpicture}\ \begin{tikzpicture}[scale=0.4]
\draw (0,0) rectangle (10,10);
\draw[fill=black!20] (7,7) rectangle (10,10);
\draw[fill=black!20] (0,4) rectangle (7,7);
\draw[fill=black!20] (7,0) rectangle (10,4);
\draw[fill=blue!20] (7,4) rectangle (10,7);
\draw[step=1.0,black] (7,4) grid (10,7);
\fill (7.5,4.5) circle (7pt);
\fill (9.5,6.5) circle (7pt);
\draw[loosely dotted,thick,line width = 0.5mm] (7.5,4.5) -- (9.5,6.5);
\draw [decorate, decoration={brace, mirror, amplitude=6pt},thick] (10.5,7.1) -- (10.5,9.9);
\draw (12.75,8.5)node {$r_k-1$};
\draw [decorate, decoration={brace, mirror, amplitude=6pt},thick] (10.5,4.1) -- (10.5,6.9);
\draw (13.25,5.5)node {$l_k-l_{k-1}$};
\draw [decorate, decoration={brace, mirror, amplitude=6pt},thick] (10.5,0.1) -- (10.5,3.9);
\draw (14.25,2)node {$l_{k-1}-r_{k}+2$};
\draw [decorate, decoration={brace, mirror, amplitude=6pt},thick] (0.1,-0.5) -- (6.9,-0.5);
\draw (3.5,-1.5)node {$l_{k-1}+1$};
\draw [decorate, decoration={brace, mirror, amplitude=6pt},thick] (7.1,-0.5) -- (9.9,-0.5);
\draw (9,-1.5)node {$l_k-l_{k-1}$};
\draw (3.5,8.5)node {$A_1$};
\draw (3.5,2)node {$A_2$};
\draw (-3,5)node {$\Lambda(\mcD)=$};
\end{tikzpicture}
\caption{$\Lambda(\mcD')$ contained in $\Lambda(\mcD)$ when $\mcD$ is an increasing chain.}
\label{fig:strongchains_contain_1}
\end{figure}
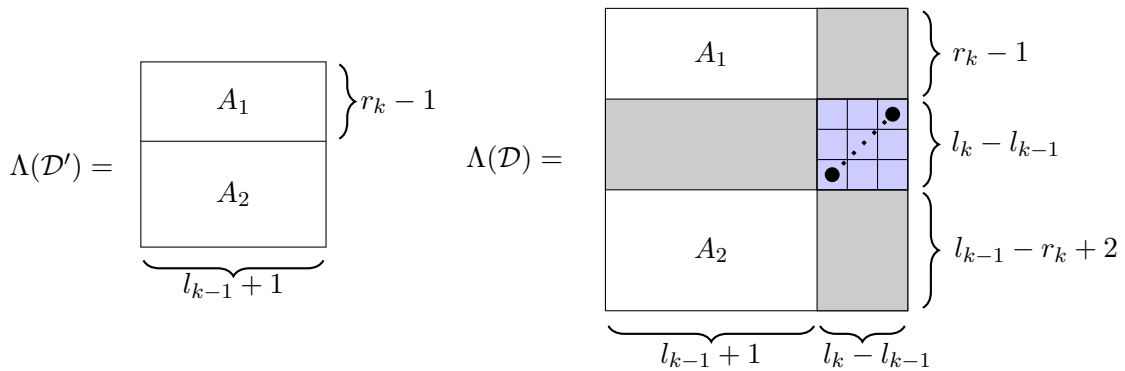


Let $\Lambda(\mcD')=w'(1)\cdots w'(l_{k-1}+1)$ and $\Lambda(\mcD)=w(1)\cdots w(n)$.  By induction, the permutation $\Lambda(\mcD')$ contains the pattern $(k+1)k12\cdots (k-1)$ along the subsequence $$w'(1)w'(2)w'(l_1+1)w'(l_2+1)\cdots w'(l_{k-1}+1).$$
We further assume by induction that $w'(1)=l_{k-1}+1$ and $w'(2)=l_{k-1}.$  Since $\mcD$ is strongly connected, $r_k\leq l_{k-1}$ and hence the nodes $(1,w(1)), (2,w(2))$ belong to region $A_2$.  This implies $w(1)=l_k+1=n$ and $w(2)=l_{k}=n-1$. 
 \Cref{prop:staircase_permutation_properties} part (2) implies $w(l_{k-1}+1)<w(l_{k-1}+2)$, which, in turn, implies that $w(l_{k-1}+1)<w(l_{k}+1).$
 Thus $\Lambda(\mcD)$ contains the pattern $(k+2)(k+1)12\cdots k$ at $$w(1)w(2)w(l_1+1)w(l_2+1)\cdots w(l_k+1)$$ with $w(1)=n$ and $w(2)=n-1$.  This completes the proof.
\end{proof}

\begin{example}
Consider the staircase diagram $\mcD=\{[1,4]\prec [3,6]\prec[6,7,8]\}$.
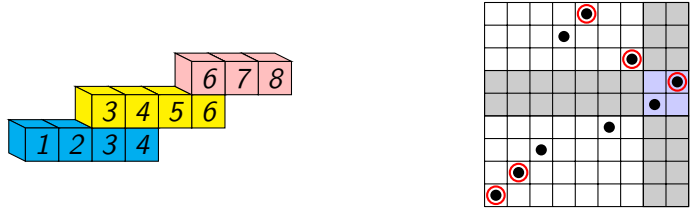
$$
\begin{tikzpicture}[scale=0.44,baseline={(0,-0.4)}]
    \ppAff{9}{{0,0, 0,0,1,1,1,1},{0,0,3,3,3,3,0},{4,4, 4,0, 0,0,0,0}}
\end{tikzpicture}\hspace{1.0in} \begin{tikzpicture}[scale=0.3,baseline={(0,0.2)}]
\draw[fill=blue!20] (7,4) rectangle (9,6);
\draw[fill=black!20] (0,4) rectangle (7,6);
\draw[fill=black!20] (7,0) rectangle (9,4);
\draw[fill=black!20] (7,6) rectangle (9,9);
\draw[step=1.0,black] (0,0) grid (9,9);
\fill (0.5,0.5) circle (7pt);\draw[thick, red]  (0.5,0.5) circle (11pt);
\fill (1.5,1.5) circle (7pt);\draw[thick, red]   (1.5,1.5) circle (11pt);
\fill (2.5,2.5) circle (7pt);
\fill (3.5,7.5) circle (7pt);
\fill (4.5,8.5) circle (7pt);\draw[thick, red]  (4.5,8.5) circle (11pt);
\fill (5.5,3.5) circle (7pt);
\fill (6.5,6.5) circle (7pt);\draw[thick, red]  (6.5,6.5) circle (11pt);
\fill (7.5,4.5) circle (7pt);
\fill (8.5,5.5) circle (7pt);\draw[thick, red]  (8.5,5.5) circle (11pt);
\end{tikzpicture}
$$
Note that $\Lambda(\mcD)=987216354$ contains $\Lambda(\mcCh_3^+)=54123$ at position sequence $(1,2,5,7,9)$.  
\end{example}

\begin{lemma}\label{lem:chain_avoid_1}
Let $k\ge 2$ and $\mcD=\{B_1,\ldots,B_k\}$ be a strongly connected staircase diagram such that $B_k$ is minimal.  Let $\mcD'=\mcD\setminus\{B_k\}$ and $\ell\geq 2$.  If $\Lambda(\mcD)$ contains $\Lambda(\mcCh^{-}_\ell)$, then $\Lambda(\mcD')$ contains $\Lambda(\mcCh^{-}_{\ell-1})$.
\end{lemma}

\begin{proof}
    Assume $S(\mcD)=[n-1]$ and let $B_k=[r,n-1]$.  Since $\mcD$ is strongly connected, we have $B_{k-1}\cap B_k=[r,m-1]$ for some $r< m<n-1$.  If $\Lambda(\mcD')$ contains $\Lambda(\mcCh^{-}_{\ell})$, then it also contains $\Lambda(\mcCh^{-}_{\ell-1})$.  Suppose $\Lambda(\mcD')$ avoids $\Lambda(\mcCh^{-}_{\ell})$.  Since $B_k$ is minimal, \Cref{eqn:B_k_min_block} together with the assumption $\Lambda(\mcD)$ contains $\Lambda(\mcCh^{-}_\ell)$ implies that every such containment must involve positions in $\Lambda(\mcD)$ from the interval $[r,r+n-m]$. Specifically, the containment must use exactly one position from $[m,m+n-r]$ for the value $\ell+2$ from $\Lambda(\mcCh^{-}_{\ell})$ as in \Cref{fig:strong_chain1}.  Thus $\Lambda(\mcD')$ contains $\Lambda(\mcCh^{-}_{\ell-1})$.
 \begin{figure}[h!]
\begin{tikzpicture}[scale=0.35,baseline={(0,-1.5)}]
\draw[fill=black!20] (0,0) rectangle (3,1);
\draw[fill=black!20] (4,0) rectangle (6,1);
\draw[fill=black!20] (3,1) rectangle (4,6);
\draw[step=1.0,black] (0,0) grid (6,6);
\fill (0.5,3.5) circle (7pt);
\fill (2.5,1.5) circle (7pt);
\fill (3.5,0.5) circle (7pt);\draw[thick, red]  (3.5,0.5) circle (11pt);
\fill (4.5,4.5) circle (7pt);
\fill (5.5,5.5) circle (7pt);
\draw[loosely dotted,thick,line width = 0.5mm] (0.5,3.5) -- (2.5,1.5);
\draw [decorate, decoration={brace, mirror, amplitude=6pt},thick] (6.5,0.1) -- (6.5,5.9);
\draw (8.75,3)node {$\ell+2$};
\draw (-3,3)node {$\Lambda(\mcCh^{-}_{\ell})=$};
\end{tikzpicture}\hspace{0.25in}
\begin{tikzpicture}[scale=0.4]
\draw (0,0) rectangle (10,10);
\draw[step=1.0,black] (4,0) grid (7,3);
\draw[fill=black!20] (0,0) rectangle (4,3);
\draw[fill=black!20] (7,0) rectangle (10,3);
\draw[fill=black!20] (4,3) rectangle (7,10);
\draw[fill=red!20] (4,0) rectangle (7,3);
\fill (4.5,0.5) circle (7pt);
\fill (6.5,2.5) circle (7pt);
\draw[loosely dotted,thick,line width = 0.4mm] (4.5,0.5) -- (6.5,2.5);
\draw[thick, red]  (4.5,0.5) circle (11pt);
\draw [decorate, decoration={brace, mirror, amplitude=6pt},thick] (10.5,3.1) -- (10.5,9.9);
\draw (11.75,6.5)node {$m$};
\draw [decorate, decoration={brace, mirror, amplitude=6pt},thick] (10.5,0.1) -- (10.5,2.9);
\draw (12.5,1.5)node {$n-m$};
\draw [decorate, decoration={brace, mirror, amplitude=6pt},thick] (0.1,-0.5) -- (3.9,-0.5);
\draw (2,-1.5)node {$r-1$};
\draw [decorate, decoration={brace, mirror, amplitude=6pt},thick] (4.1,-0.5) -- (6.9,-0.5);
\draw (5.5,-1.5)node {$n-m$};
\draw (2,6.5)node {$A_3$};
\draw (8.5,6.5)node {$A_4$};
\draw (-3,5)node {$\Lambda(\mcD)=$};
\end{tikzpicture}
\caption{The pattern $\Lambda(\mcCh^{-}_{\ell})$ contained in $\Lambda(\mcD)$.}
\label{fig:strong_chain1}
\end{figure}
\end{proof}

\begin{lemma}\label{lem:chain_avoid_2}
Let $\mcD=\{B_1,\ldots,B_k\}$ be a strongly connected staircase diagram and let $1< d\leq k$ and $\ell\geq 1$.  Define the subdiagrams 
$$\mcD'=\{B_1,\ldots, B_{d-1}\}\quad\text{and}\quad  \mcD''=\{B_d,\ldots,B_k\}.$$
Suppose that $B_d$ is maximal in $\mcD$ and that $\mcD''$ is a decreasing chain.  If $\Lambda(\mcD)$ contains $\Lambda(\mcCh^{-}_\ell)$, then either $\Lambda(\mcD')$ contains $\Lambda(\mcCh^{-}_{\ell})$ or  $\Lambda(\mcD'')$ contains $\Lambda(\mcCh^{-}_{\ell})$.
\end{lemma}

\begin{proof}
We first consider the case where $d=k$ and hence $\mcD''=\{B_d\}$.  Note that $\Lambda(\mcD'')=u_{B_d}$ which contains $\Lambda(\mcCh^{-}_\ell)$ if and only if $\ell=1$.  So we assume that $\ell\geq 2$.  Since $B_d$ is maximal, the matrix of $\Lambda(\mcD)$ has the following form as in \Cref{fig:overlap_proof_1}, where the submatrix $\Lambda(\mcD')$ is given by regions $A_1$ and $A_2$, and region $A_3$ is an anti-diagonal square.

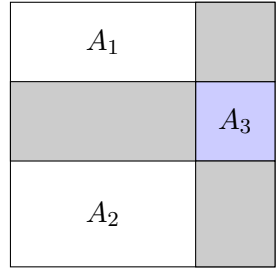
\begin{figure}[h!]
\begin{tikzpicture}[scale=0.35]
\draw (0,0) rectangle (10,10);
\draw[fill=blue!20] (7,4) rectangle (10,7);
\draw[fill=black!20] (0,4) rectangle (7,7);
\draw[fill=black!20] (7,0) rectangle (10,4);
\draw[fill=black!20] (7,7) rectangle (10,10);
\draw (3.5,8.5)node {$A_1$};
\draw (3.5,2)node {$A_2$};
\draw (8.5,5.5)node {$A_3$};
\end{tikzpicture}
\caption{The matrix of $\Lambda(\mcD)$ with $B_d$ maximal.  The submatrix of $\Lambda(\mcD')$ is given by regions $A_1$ and $A_2$.}\label{fig:overlap_proof_1}
\end{figure}

 Lemma \ref{lem:regions_A_2_4_decreasing} says that region $A_2$ is decreasing.  If $\Lambda(\mcD)$ contains $\Lambda(\mcCh^{-}_\ell)$, then any such containment must happen within region $A_1$ since region $A_3$ is also decreasing.  Hence $\Lambda(\mcD')$ contains $\Lambda(\mcCh^{-}_\ell)$.  Now suppose that $d<k$ and hence $\mcD''$ is a decreasing chain with more than one block.  The matrix of $\Lambda(\mcD)$ is given in \Cref{fig:overlap_proof_2}.

\begin{figure}[h!]
\begin{tikzpicture}[scale=0.35]
\draw (1,-3) rectangle (14,10);
\draw[fill=blue!20] (12,5) rectangle (14,7);
\draw[fill=blue!20] (7,3) rectangle (9,5);
\draw[fill=black!20] (1,3) rectangle (7,7);
\draw[fill=black!20] (7,0) rectangle (14,3);
\draw[fill=black!20] (7,7) rectangle (14,10);
\draw[fill=red!20] (9,-3) rectangle (12,0);
\draw[fill=black!20] (9,0) rectangle (12,10);
\draw[fill=black!20] (1,-3) rectangle (14,0);
\draw[fill=red!20] (9,-3) rectangle (12,0);
\draw (4,8.5)node {$A_1$};
\draw (4,1.5)node {$A_2$};
\draw (8,4)node {$A_3'$};
\draw (13,6)node {$A_3''$};
\draw (13,6)node {$A_3''$};
\draw (10.5,-1.5)node {$A_4$};
\end{tikzpicture}
\caption{The matrix of $\Lambda(\mcD)$ with $d<k$ and $\{B_d,\ldots,B_k\}$ decreasing.}\label{fig:overlap_proof_2}
\end{figure}
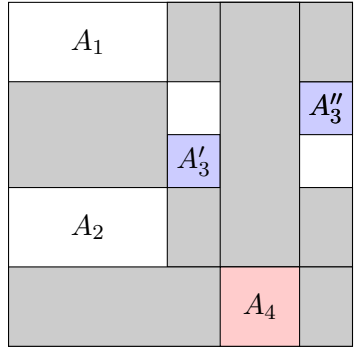

Here, region $A_4$ corresponds to appending the sequence of blocks $B_{d+1}\succ\cdots\succ B_k$ to $\mcD'\cup \{B_d\}$.
The region $A_3$ (from block $B_d$) splits into two regions, $A_3'$ and $A_3''$.  Since $A_3$ is anti-diagonal, each of $A_3', A_3''$ are also anti-diagonal with $A_3''$ strictly above and to the right of $A_3'$.  Also note that region $A_4$ is strictly to the right of regions $A_1$ and $A_2$ since $B_d\succ\cdots\succ B_k$ is a decreasing sequence with $B_d$ maximal in $\mcD$.  

Now suppose that $\Lambda(\mcD)$ contains $\Lambda(\mcCh^{-}_\ell)$.  Then either $\Lambda(\mcD)$ is contained in region $A_1$ or contained in regions $A_4$ and $A_3''$ since regions $A_2$ and $A_3'$ are decreasing.  But $\Lambda(\mcD')$ contains region $A_1$ and $\Lambda(\mcD'')$ contains the region given by $A_4$ and $A_3''$ (in fact, $\Lambda(\mcD'')$ contains the region given by $A_3'$, $A_4$, and $A_3''$).  Hence, at least one of $\Lambda(\mcD'),\Lambda(\mcD'')$ contains $\Lambda(\mcCh^{-}_\ell)$.
\end{proof}

\begin{proof}[Proof of \Cref{thm:strong_chains}]
Let $\mcD=\{B_1,\ldots,B_k\}$ be a staircase diagram and suppose that $\mcD$ contains a strongly connected $\ell$-chain.  By \Cref{prop:subdiagram_containment} and \Cref{lem:chain_containment}, if the chain is increasing, then $\Lambda(\mcD)$ contains $\Lambda(\mcCh^+_\ell)$, and if the chain is decreasing, then $\Lambda(\mcD)$ contains $\Lambda(\mcCh^-_\ell)$.   

Conversely, suppose that $\Lambda(\mcD)$ contains $\Lambda(\mcCh^-_\ell)$.  We will show that $\mcD=\{B_1,\ldots,B_k\}$ contains a decreasing $\ell$-chain by induction on the number of decreasing runs in $\mcD$.   Since $\Lambda(\mcCh^{-}_\ell)$ is strongly indecomposable, it suffices to assume $\mcD$ is a strongly connected staircase diagram by \Cref{lem:wk_decomposable_patterns} and \Cref{prop:staircase_wk_decomp}.    First, suppose that $\mcD$ is a single decreasing chain.  If $k<\ell$, then repeated applications of \Cref{lem:chain_avoid_1} imply that $\Lambda(\{B_1\})$ contains $\Lambda(\mcCh^-_2)=4312$ which is a contradiction.  Hence $k\geq \ell$ and $\mcD$ contains a decreasing $\ell$-chain.  

If $\mcD$ is not a decreasing chain, then let $B_d$ denote the maximal block in $\mcD$ with largest index and consider the subdiagrams $\mcD':=\{B_1,\ldots,B_{d-1}\}$ and $\mcD'':=\{B_d,\ldots,B_k\}$.  In particular, $\mcD''$ is a decreasing chain.  Since $\Lambda(\mcD)$ contains $\Lambda(\mcCh^-_\ell)$, \Cref{lem:chain_avoid_2} implies at least one of $\Lambda(\mcD')$ or $\Lambda(\mcD'')$ contains $\Lambda(\mcCh^-_\ell)$.  If $\Lambda(\mcD'')$ contains $\Lambda(\mcCh^-_\ell)$, then $\mcD''$ must be a decreasing chain of at least $\ell$ blocks and hence $\mcD$ contains a decreasing $\ell$-chain.  Otherwise, $\Lambda(\mcD')$ contains $\Lambda(\mcCh^-_\ell)$ and, by induction, the subdiagram $\mcD'$ contains a decreasing $\ell$-chain.  In either case, $\mcD$ contains a decreasing $\ell$-chain.

Finally, if $\Lambda(\mcD)$ contains $\Lambda(\mcCh^+_\ell)$, then $\mcD$ contains an increasing $\ell$-chain by \Cref{prop:staircase_permutation_properties} part (4) and the preceding arguments.  This proves the theorem.  
\end{proof}

As a corollary of \Cref{thm:overlaps} and \Cref{thm:strong_chains}, we get the following characterization of polished permutations from \Cref{thm:GG_polished} in terms of staircase diagrams.  We remark this corollary can also be deduced from results in \cite{GG20}.
\begin{cor}
\label{cor:polished_staircases}
The following are equivalent:
\begin{enumerate}
    \item The staircase diagram $\mcD$ avoids strongly connected $3$-chains and $2$-overlaps.
    \item The permutation $\Lambda(\mcD)$ is polished.
\end{enumerate}    
\end{cor}

\begin{proof}
By \Cref{thm:strong_chains} and \Cref{thm:overlaps}, staircase diagrams avoiding $2$-overlaps and 
strongly connected $3$-chains correspond to smooth permutations avoiding 45321, 54312, 34521, 54123.
$$
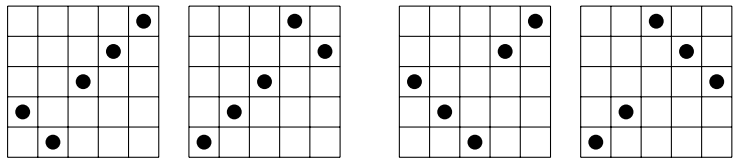
\begin{tikzpicture}[scale=0.4]
\draw[step=1.0,black] (0,0) grid (5,5);
\fill (0.5,1.5) circle (7pt);
\fill (1.5,0.5) circle (7pt);
\fill (2.5,2.5) circle (7pt);
\fill (3.5,3.5) circle (7pt);
\fill (4.5,4.5) circle (7pt);
\end{tikzpicture}\quad
\begin{tikzpicture}[scale=0.4]
\draw[step=1.0,black] (0,0) grid (5,5);
\fill (0.5,0.5) circle (7pt);
\fill (1.5,1.5) circle (7pt);
\fill (2.5,2.5) circle (7pt);
\fill (3.5,4.5) circle (7pt);
\fill (4.5,3.5) circle (7pt);
\end{tikzpicture}\qquad
\begin{tikzpicture}[scale=0.4]
\draw[step=1.0,black] (0,0) grid (5,5);
\fill (0.5,2.5) circle (7pt);
\fill (1.5,1.5) circle (7pt);
\fill (2.5,0.5) circle (7pt);
\fill (3.5,3.5) circle (7pt);
\fill (4.5,4.5) circle (7pt);
\end{tikzpicture}\quad
\begin{tikzpicture}[scale=0.4]
\draw[step=1.0,black] (0,0) grid (5,5);
\fill (0.5,0.5) circle (7pt);
\fill (1.5,1.5) circle (7pt);
\fill (2.5,4.5) circle (7pt);
\fill (3.5,3.5) circle (7pt);
\fill (4.5,2.5) circle (7pt);
\end{tikzpicture}$$
These are exactly the permutations patterns that appear in \Cref{thm:GG_polished}.
\end{proof}

\section{Generating functions and staircase diagrams}\label{sec:GF_calc}

In this section, we develop techniques to study the generating functions of smooth permutations which avoid strongly indecomposable permutation patterns.  We also define and study pattern avoidance of ``chain indecomposable" permutations.  We will see that these patterns behave nicely with respect to the strongly connected chain structure of staircase diagram.  As an application, we prove \Cref{thm:polished_GF} and \Cref{thm:semi-polished_GF12}.  


We say $\mcD$ is \textbf{fully supported} on $[n]$ if $S(\mcD)=[n]$.  Let $Q$ denote a set permutation patterns.  We say a staircase diagram \textbf{$\mcD$ avoids the set $Q$} if $\Lambda(\mcD)$ avoids every permutation pattern in $Q$.  Define the following coefficients (dependent on $Q$):
\begin{align*}
z_n&:=\#\{\text{$\mcD$ on $[n]$}\ |\ \text{$\mcD$ avoids $Q$}\}\\
z'_n&:=\#\{\text{$\mcD$ on $[n]$}\ |\ \text{$\mcD$ is fully supported and avoids $Q$}\}\\
\bar z_n&:=\#\{\text{$\mcD$ on $[n]$}\ |\ \text{$\mcD$ is fully supported, strongly connected, and avoids $Q$}\},
\end{align*}
and corresponding generating functions:
\[Z_Q(x):=\sum_{n\geq 0} z_n\, x^n,\qquad Z_Q'(x):=\sum_{n\geq 1} z'_n\, x^n,\qquad \overline{Z}_Q(x):=\sum_{n\geq 1} \bar z_n\, x^n\]
where we set $z_0:=1$.  The next proposition relates $Z(x), Z'(x)$, and $\overline{Z}(x)$.  For the proof, we follow similar arguments used in \cite{RS17} and \cite{RS18} to enumerate staircase diagrams (without pattern avoidance, so $Q=\varnothing$).


\begin{theorem}\label{thm:str_conn_2_all_SD}
Suppose that $Q$ is a set of strongly indecomposable permutations.  Then
$$Z_Q(x)=\frac{1-2\overline{Z}_Q(x)}{1-x+(x-2)\overline{Z}_Q(x)}.$$
\end{theorem}

\begin{proof}
\Cref{lem:wk_decomposable_patterns} and \Cref{prop:staircase_wk_decomp} imply that, to enumerate staircase diagrams that avoid strongly indecomposable permutations, it suffices to consider staircase diagrams that are strongly connected and fully supported on $[n]$.  Each fully supported staircase diagram decomposes as a strongly connected staircase diagram and a fully supported staircase diagram on a smaller connected support.  
For example: 
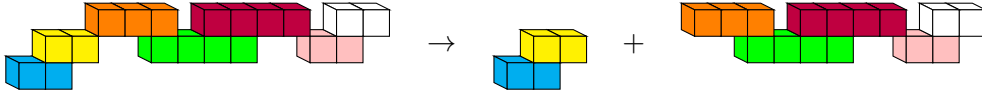
\[
\begin{tikzpicture}[scale=0.35,baseline={(0,0.5)}]
            \ppsansnumbers{12}{ {0,0,0,0,0,0,0,0,0,0,0,0,1,1},{0,4,4,0,0,2,2,2,2,0,0,3,3,0},{7,7,0,6,6,6,6,0,5,5,5,0,0,0}}
    \end{tikzpicture}\quad  \rightarrow \quad 
\begin{tikzpicture}[scale=0.35,baseline={(0,0.5)}]
            \ppsansnumbers{12}{ {0,1,1},{3,3,0},{0,0,0}}
\end{tikzpicture}\quad + \quad
\begin{tikzpicture}[scale=0.35,baseline={(0,0.5)}]
            \ppsansnumbers{12}{ {0,0,0,0,0,0,0,0},{0,4,4,0,0,2,2,2,2,0,0},{7,7,0,6,6,6,6,0,5,5,5}}
    \end{tikzpicture}\]
Since $\mcD$ avoids $Q$ if and only if each component avoids $Q$, we have the recursion:
\[z'_n=\sum_{k=1}^{n} \bar z_{k}\cdot 2z'_{n-k}.\]
Thus 
\begin{equation}\label{eqn:Z_bar_2_Z_prime}
Z'_Q(x)=\overline{Z}_Q(x)+2\overline{Z}_Q(x)\cdot Z'_Q(x)\quad \text{giving}\quad Z'_Q(x)=\frac{\overline{Z}_Q(x)}{1-2\overline{Z}_Q(x)}.
\end{equation}
Furthermore, any staircase diagram is a disjoint union of fully supported staircase diagrams.  Since a strongly indecomposable permutation is also indecomposable, \Cref{lem:decomposable} implies 
\begin{equation}\label{eqn:Z_prime_2_Z}
Z_Q(x)=\frac{1+Z'(x)}{1-x-xZ'_Q(x)}.
\end{equation}
Combining \Cref{eqn:Z_bar_2_Z_prime} and \Cref{eqn:Z_prime_2_Z} completes the proof.
\end{proof}

\subsection{Chain indecomposable permutations}
We consider a further refinement of staircase diagrams into strongly connected chains which alternate between increasing and decreasing pieces.  One issue is that extremal blocks often belong to more than one chain.  So to construct a well-defined decomposition into strongly connected chains, we make the following definition.  

\begin{definition}
Let $\mcD=\{B_1,\ldots,B_k\}$ be strongly connected chain with $k\geq 2$.  Define 
$$\mcBr(\mcD):=\{B_1,\ldots, B_{k-1},B_k\cap B_{k-1}\}$$
with the same order structure as $\mcD$ (either increasing or decreasing).
If $\mcE=\mcBr(\mcD)$ for some staircase diagram $\mcD=\{B_1,\ldots,B_k\}$, we say $\mcE$ is a \textbf{broken staircase diagram (or broken staircase)} on $[r]$ where $r$ denotes the right endpoint of $B_{k-1}$.
\end{definition}
For example,
$$\mcD=\
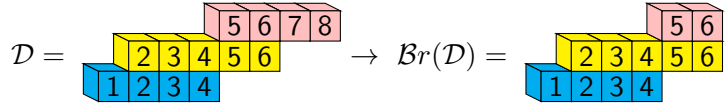
\begin{tikzpicture}[scale=0.4,baseline={(0,0.5)}]
    \ppAff{9}{{0,0, 0,0,1,1,1,1},{0,0,3,3,3,3,3},{4,4,4,4, 0,0,0}}
\end{tikzpicture}\rightarrow\  \mcBr(\mcD)=\
\begin{tikzpicture}[scale=0.4,baseline={(0,0.5)}]
    \ppAff{9}{{0,0, 0,0,1,1,1,1},{0,0,3,3,3,3,3},{0,0,4,4, 0,0,0}}
\end{tikzpicture}
$$
Note that broken staircases are not staircase diagrams and if $\mcE=\{B_1,\ldots, B_k\}$ is broken staircase, then $B_k\subset B_{k-1}$.  Broken staircase diagrams were used in \cite{RS18} to enumerate smooth affine permutations and were defined for any chain, not necessarily strongly connected. 


Observe that $\mcE=\mcBr(\mcD)$ is not uniquely determined by $\mcD$, so given a broken staircase diagram $\mcE$, define the fiber over $\mcE$ as  
$$\mcBr^{-1}(\mcE):=\{\text{$\mcD$ is a strongly connected chain}\ |\ \mcBr(\mcD)=\mcE\ \}.$$
If $\mcE$ is broken staircase on $[r]$, then $\mcBr^{-1}(\mcE)$ contains a unique staircase diagram on $[n]$ for each $n>r$.  

Next we define an operator on staircase diagrams that deletes the first entry of the first block if allowed.  Let $\mcD=\{B_1,\ldots,B_k\}$ be a staircase diagram and let $l$ denote the left endpoint of $B_1$. Define the operator 

$$\mcFr(\mcD):=\begin{cases}
\{B_1\setminus\{l\}, B_2,\ldots ,B_k\}& \text{if $B_1\setminus\{l\}\not\subset B_2$}\\
\mcD & \text{if $B_1\setminus\{l\}\subseteq B_2$}
\end{cases}.$$
Unlike the broken staircase operator $\mcBr$, the operator $\mcFr$ gives a valid staircase diagram.  For $j\geq 0$, let $\mcFr^j(\mcD):=\underbrace{\mcFr\circ \cdots \circ\mcFr}_{\text{$j$ times}}(\mcD)$ and define the corresponding fiber over $\mcD$ as
$$\mcFr^{-1}(\mcD):=\{\mcD'\ \text{is a staircase diagram}\ | \ \mcFr^j(\mcD')=\mcD\ \text{from some $j\geq 0$}.\}$$
If $X$ is a set of staircase diagrams, then define $\mcFr^{-1}(X):=\bigcup_{\mcD\in X}\mcFr^{-1}(\mcD)$. 

\begin{definition}\label{def:broken_contains_avoids}
We say a broken staircase diagram \textbf{$\mcE$ contains the pattern $p$} if there exists $\mcD\in\mcFr^{-1}(\mcBr^{-1}(\mcE))$ such that $\Lambda(\mcD)$ contains $p$.  Otherwise, we say \textbf{$\mcE$ avoids the pattern $p$}. 
\end{definition}

By definition, a broken staircase $\mcE$ avoiding $p$ means that $\Lambda(\mcD)$ avoids $p$ where $\mcD$ is any ``extension" of the first and last blocks of $\mcE$.
 

\begin{definition}\label{def:chain-indecomposable}
 We say a permutation $p$ is \textbf{chain indecomposable} if the following hold:
 \begin{enumerate}
 \item For any staircase diagram $\mcD$, if $\Lambda(\mcD)$ contains $p$, then there exists a strongly connected chain $\mcD'\subseteq \mcD$ such that $\Lambda(\mcD')$ contains $p$. 
 \item For any strongly connected chain $\mcD$, if $\Lambda(\mcD)$ avoids $p$, then $\mcBr(\mcD)$ avoids $p$.
 \end{enumerate}
\end{definition}

The motivation for \Cref{def:chain-indecomposable} is to characterize permutations whose avoidance by $\Lambda(\mcD)$ is captured by the combinatorial structure of chains in $\mcD$.  

\begin{lemma}\label{lem:chain-indecomp-smooth}
    Let $p$ be smooth permutation with $p=\Lambda(\mcD)$.  If $p$ is chain indecomposable, then $\mcD$ is a strongly connected chain.  As a consequence, $p$ is strongly indecomposable.
\end{lemma}

\begin{proof}
Since $p$ is chain indecomposable and $p$ contains itself, there exists a strongly connected chain $\mcD'\subseteq \mcD$ such that $\Lambda(\mcD')$ contains $p=\Lambda(\mcD)$.  But $\Lambda(\mcD)$ cannot contain a permutation of a staircase diagram with strictly smaller support. Hence $\mcD'=\mcD$.
\end{proof}

Note that the converse of \Cref{lem:chain-indecomp-smooth} is false due to the second condition in \Cref{def:chain-indecomposable}.  For example, consider $p=\Lambda(\{[1,2]\prec [2,4]\})=54132$ and $\mcD=\{[1,2]\prec[2,3]\}$. Then $\Lambda(\mcD)=4312$ avoids $p$, but $\mcBr(\mcD)$ contains $p$ since $\{[1,2]\prec[2,4]\}$ belongs to $\mcBr^{-1}(\mcBr(\mcD))$.  Thus $p$ is the permutation of a strongly connected chain, but it is not chain indecomposable.

By \Cref{thm:overlaps} and \Cref{thm:strong_chains}, the permutations $\Lambda(\mcOv^{\pm}(q))$ and $\Lambda(\mcCh^{\pm}(\ell))$ are chain indecomposable.  In particular, the property of a staircase diagram avoiding $q$-overlaps and/or strongly connected $\ell$-chains reduces to a property on strongly connected chains of the staircase diagram.  In particular, the avoidance/containment of $\mcOv^{\pm}(q)$ and $\mcCh^{\pm}(\ell)$ is invariant on each fiber $\mcFr^{-1}(\mcBr^{-1}(\mcE))$ and hence the second condition in \Cref{def:chain-indecomposable} is satisfied.

Let $\mcD=\{B_1,\ldots,B_k\}$ be a strongly connected staircase diagram on $[n]$.  Consider the sequence of maximal chains in $\mcD$ ordered by the induced linear order on blocks:
\[\mcD=\mcD_1\cup\cdots \cup D_K.\]
For each maximal chain, $\mcD_i=\{B_x,\ldots,B_y\}$ where $i<K$, define the broken staircase diagram 
\[\mcE_i:=\{B_{x}\setminus B_{x-1}, B_{x+1}, B_{x+2},\ldots, B_{y-1}, B_y\cap B_{y-1}\}.\]
Here we set $B_0:=\varnothing$.  When $i=K$, we define $\mcE_K$ by replacing $B_x$ with $B_{x}\setminus B_{x-1}$, but leaving $B_y=B_k$ unchanged.  In particular $\mcE_K$ is an unbroken staircase diagram with at least two blocks.  Following \cite{RS18}, we consider the map \[\mcD\mapsto (\mcE_1,\ldots,\mcE_{K-1}, \mcE_K)\] which decomposes $\mcD$ to a disjoint union of broken staircases followed by a strongly connected chain.  For example:
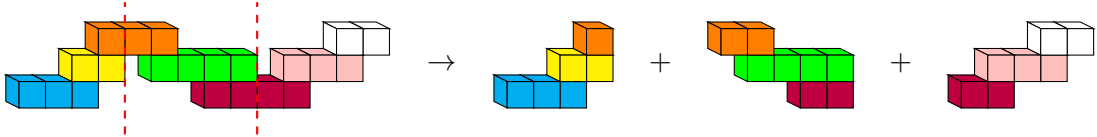
\[
\begin{tikzpicture}[scale=0.35,baseline={(0,0.5)}]
            \ppsansnumbers{12}{ 
            {0,0,0,6,6,6,6,0,0,0,0,1,1,1},
            {0,4,4,4,0,2,2,2,2,0,3,3,0,0},
            {7,7,0,0,0,0,0,0,5,5,5,0,0,0}}
    \draw[dashed,red, thick] (-4,-1)--(-4,4);
    \draw[dashed,red,thick] (-9,-1)--(-9,4);
    \end{tikzpicture}\quad  \rightarrow \quad 
\begin{tikzpicture}[scale=0.35,baseline={(0,0.5)}]
            \ppsansnumbers{12}{ 
            {0,1,1,1},
            {3,3,0,0},
            {5,0,0,0}}
\end{tikzpicture}\quad + \quad
\begin{tikzpicture}[scale=0.35,baseline={(0,0.5)}]
            \ppsansnumbers{12}{ 
            {6,6,0,0,0},
            {2,2,2,2,0},
            {0,0,0,5,5}}
    \end{tikzpicture}\quad + \quad
\begin{tikzpicture}[scale=0.35,baseline={(0,0.5)}]
            \ppsansnumbers{12}{ 
            {0,0,0,6,6},
            {0,4,4,4,0},
            {7,7,0,0,0}}
    \end{tikzpicture}
\]

\begin{proposition}\label{prop:broken_decomposition}
    Let $\mcD$ be a strongly connected staircase diagram and let $p$ be a chain indecomposable permutation.  Consider the decomposition $\mcD\mapsto (\mcE_1,\ldots,\mcE_K)$ as above.  Then $\Lambda(\mcD)$ avoids $p$ if and only if $\mcE_i$ avoids $p$ for each $1\leq i \leq K$.
\end{proposition}

\begin{proof}
Let $\mcD_1,\ldots,\mcD_K$ denote the maximal chains of $\mcD$ corresponding to $\mcE_1,\ldots,\mcE_K$.  If $\Lambda(\mcD)$ avoid $p$, then each of $\Lambda(\mcD_i)$ must avoid $p$ by \Cref{prop:subdiagram_containment}.  Using \Cref{lem:containment_expand} and the assumption that $p$ is chain indecomposable, it follows that each of the corresponding broken staircases $\mcE_i$ avoids $p$.

Conversely, suppose that each $\mcE_i$ avoids $p$.  Since $\mcD_i\in \mcFr^{-1}(\mcBr^{-1}(\mcE_i))$, we get that the permutation $\Lambda(\mcD_i)$ avoids $p$.  This implies $\Lambda(\mcD')$ avoids $p$ for every chain $\mcD'\subseteq\mcD$ since each chain in $\mcD$ is in some maximal $\mcD_i$.  Again, since $p$ in chain indecomposable, we have that $\Lambda(\mcD)$ avoids $p$.
\end{proof}

Let $Q$ denote a set of permutation patterns and define
$$z^+_n:=\#\left\{\text{$\mcD$ on $[n]$}\ \Big| \ \substack{\displaystyle\mcD\ \text{is a fully supported, strongly connected}\\ \displaystyle \text{increasing chain and avoids}\ Q}\right\}$$ with generating function 
$$Z^+_Q(x):=\sum_{n\geq 1} z^+_n\, x^n.$$
Recall that 
\[\bar z_n=\#\{\text{$\mcD$ on $[n]$}\ |\ \text{$\mcD$ is fully supported, strongly connected, and avoids $Q$}\}\]
with generating function \[\overline{Z}_Q(x)=\sum_{n\geq 1}\bar z_n\, x^n.\]  We say a set of permutations $Q$ is \textbf{inverse invariant} if $Q=Q^*:=\{p^{-1}\ |\ p\in Q\}$.

\begin{theorem}\label{thm:connected2chains_gf}
    Let $Q$ be an inverse invariant set of chain indecomposable permutations.  Then 
    \[\overline{Z}_Q(x)=\frac{xZ^+_Q(x)}{2x-(1-x)Z^{+}_Q(x)}.\]
\end{theorem}

\begin{proof}
First note that since $Q$ is inverse invariant, we have $z^+_n=z^-_n$ where $z^-_n$ denotes the analogous number of strongly connected decreasing chains which avoid $Q$.  Let $\mcD$ denote a fully supported, strongly connected staircase diagram on $[n]$.  Then either $\mcD$ is a chain (increasing or decreasing), or $\mcD$ decomposes as $(\mcE_1,\ldots,\mcE_K)$ as in \Cref{prop:broken_decomposition} with $K\geq 2$.  Let $z_k^b$ denote the number of increasing broken staircases on $[k]$ that avoid $Q$. Since $Q$ is inverse invariant, $z_k^b$ also equals the number of decreasing broken staircases on $[k]$ that avoid $Q$. By \Cref{prop:broken_decomposition}, we have that 
\begin{equation}\label{eqn:broken_decomposition}
\bar z_n=(2z^+_n-1)+\sum_{k=2}^{n-3}z_k^b\cdot (\bar z_{n-k}-1).
\end{equation}
Observe that the bounds on the sum above come from the fact that broken staircases have minimum support size $k=2$ and that the staircase diagram $\mcE_K$ must have at least two blocks and hence $n-k\geq 3$.  Next, we claim that $z_k^b=z^+_{k+1}-z^{+}_k$.
Let $\mcD'=\{B_1,\ldots, B_{k}\}$ be a fully supported increasing chain on $[k+1]$ and define 
\[\phi(\mcD'):=\{B_1,\ldots, B_{k-1}, B_{k}\setminus\{k+1\}\}.\]
Either $\phi(\mcD')$ is a broken staircase on $[k]$, or a strongly connected chain on $[k]$.  Moreover the map $\phi$ is a bijection from increasing chains on $[k+1]$ to the disjoint union of increasing broken staircases and increasing chains on $[k]$.  Thus the claim is proved.
Now \Cref{eqn:broken_decomposition} becomes 
\begin{equation}\label{eqn:broken_decomposition2}
\bar z_n=(2z^+_n-1)+\sum_{k=2}^{n-3}(z^+_{k+1}-z^{+}_k)\cdot (\bar z_{n-k}-1)
\end{equation} and thus
\[\overline{Z}(x)=\left(2Z^+(x)-\frac{x}{1-x}\right)+\left(\frac{(1-x)Z^+(x)-x}{x}\right)\cdot\left(\overline{Z}(x)-\frac{x}{1-x}\right).\]
Solving for $\overline{Z}(x)$ proves the theorem.
\end{proof}

As with \Cref{thm:str_conn_2_all_SD}, the proof of \Cref{thm:connected2chains_gf} is modeled after similar arguments found in previous papers on enumerating staircase diagrams such as \cite{Az23, RS17, RS18}.  We summarize \Cref{thm:str_conn_2_all_SD} and \Cref{thm:connected2chains_gf} with the following corollary.

\begin{corollary}\label{cor:GF_summary}
    Let $Q$ be a set of permutation patterns and let $Z_Q(x)$ denote the generating function for the number of staircase diagrams on $[n]$ that avoid $Q$ (or equivalently, smooth permutations in $\mfS_{n+1}$ that avoid $Q$).
    \begin{enumerate}
    \item If $Q$ is a strongly indecomposable set, then $Z_Q(x)$ is determined by $\overline{Z}_Q(x)$, the generating function for the number of strongly connected staircase diagrams that avoid $Q$.  Furthermore, $Z_Q(x)$ is rational if and only if $\overline{Z}_Q(x)$ is rational.

    \item If $Q$ is an inverse invariant and chain indecomposable set, then $Z_Q(x)$ is determined by $Z^+_Q(x)$, the generating function for the number of strongly connected, increasing chains that avoid $Q$.  Furthermore, $Z_Q(x)$ is rational if and only if $Z^+_Q(x)$ is rational.
    \end{enumerate}
\end{corollary}

\subsection{Calculating generating functions}

In this section, we apply \Cref{thm:str_conn_2_all_SD}, \Cref{thm:connected2chains_gf}, and \Cref{cor:GF_summary} to various sets of permutations $Q$ involving the permutations $\Lambda(\mcOv^{\pm}(q))$ and $\Lambda(\mcCh^{\pm}(\ell))$ from \Cref{sec:sd_patterns}.  Note that these permutations are both strongly indecomposable and chain indecomposable, so the theory applies.

We first consider the set of polished permutations from \Cref{thm:GG_polished}: \[Q_P:=\{\Lambda(\mcOv^{\pm}(2)), \Lambda(\mcCh^{\pm}(3))\}=\{45321, 54312,
34521, 54123\}\]
By \Cref{cor:polished_staircases}, a staircase diagram $\mcD$ avoids $Q_P$ if and only if $\mcD$ avoids strongly connected $3$-chains and $2$-overlaps.  Define
\[\bar p_n=\#\{\text{$\mcD$ on $[n]$}\ |\ \text{$\mcD$ is fully supported, strongly connected, and avoids $Q_P$}\}\]
\begin{proposition}\label{prop:P_calc}
If $n\geq 2$, then $\bar p_n=2F_{n-1}-1$ where $F_n$ is the $n$-th Fibonacci number.  It follows that
\[\sum_{n\geq 1}\bar p_n\, x^n=\frac{x-x^2+x^3}{(1-x)(1-x-x^2)}.\]
\end{proposition}
\begin{proof}
Let $\mcD=\{B_1,\ldots,B_k\}$ be a strongly connected staircase diagram and let $\alpha_i:=|B_i|$.  If $\mcD$ avoids $Q_P$, then it consists of alternating $2$-chains with $1$-overlaps between blocks.  If $k>1$, then the map $$\mcD\rightarrow (\alpha_1-1,\alpha_2-2,\alpha_3-2,\ldots,\alpha_{k-3}-2,\alpha_{k-1}-2,\alpha_k-1)$$
is a two-to-one map between strongly connected staircase diagrams on $[n]$ with $k$ blocks that avoid $Q_P$ and $k$-compositions of $(n-k+1)$.  Summing over $k$ gives $\bar p_n=2F_{n-1}-1$. For example:

$$\begin{tikzpicture}[scale=0.35,baseline={(0,0.25)}]
    \ppsansnumbers{15}{
    {0,8,4,4,8,0,0,8,2,2,2,8,0,0},
    {5,8,0,0,8,3,3,8,0,0,0,8,1,1}};
    \end{tikzpicture}\quad \mapsto\quad  (2,3,2,2,1)$$ 

\smallskip
    
\noindent Here we use the fact that $\displaystyle \sum_{k\geq 0}\binom{n-k}{k}=F_{n}$.
\end{proof}

Next we consider 
\[Q_B:=\{\Lambda(\mcCh^{+}(3)), \Lambda(\mcCh^{-}(3))\}=\{34521, 54123\}\]
and define 
\[\bar b_n:=\#\{\text{$\mcD$ on $[n]$}\ |\ \text{$\mcD$ is fully supported, strongly connected, and avoids $Q_B$}\}.\]
By \Cref{thm:strong_chains}, $\mcD$ avoids $Q_B$ if and only if $\mcD$ avoids connected $3$-chains. 

\begin{proposition}\label{prop:B_calc}
We have that $\bar b_1=1$ and if $n\geq 2$, then $\bar b_n=2^{n-1}-1$.  It follows that
\[\sum_{n\geq 1}\bar b_n\, x^n=\frac{x-2x^2+2x^3}{(1-2x)(1-x)}.\]
\end{proposition}
\begin{proof}
The fact that $\bar b_1=1$ is clear.  Let $\mcD=\{B_1,\ldots,B_k\}$ be a strongly connected staircase diagram on $[n]$ and let $\alpha_i:=|B_i|$.  If $\mcD$ avoids $Q_B$, then it consists of alternating $2$-chains with non-empty overlaps between blocks. Define $\beta_i=|B_i\cap B_{i+1}|$ for $i<k$ with $\beta_k=0$ and $\gamma_i=\alpha_i-(\beta_i+\beta_{i-1})$ (here we set $\beta_0=0$). If $k>1$, then the map 
$$\mcD\rightarrow (\gamma_1,\beta_1,\gamma_2,\beta_2\ldots,\gamma_k)$$
is a two-to-one map between strongly connected staircase diagrams on $[n]$ with $k$ blocks that avoid $Q_B$ and $(2k-1)$-compositions of $n$.  Summing over $k$ gives $\bar b_n=2^{n-1}-1$.  For example:

$$
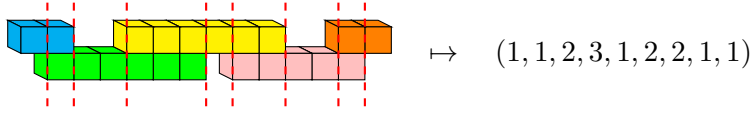
\begin{tikzpicture}[scale=0.35,baseline={(0,0.25)}]
    \ppsansnumbers{15}{
    {0,4,4,4,4,4,0,2,2,2,2,2,2,0},
    {5,5,0,0,3,3,3,3,3,3,0,0,1,1}};
    \draw[dashed,red, thick] (0,-1)--(0,3);
    \draw[dashed,red,thick] (-1,-1)--(-1,3);
    \draw[dashed,red,thick] (-3,-1)--(-3,3);
    \draw[dashed,red,thick] (-5,-1)--(-5,3);
    \draw[dashed,red,thick] (-6,-1)--(-6,3);
    \draw[dashed,red,thick] (-9,-1)--(-9,3);    
    \draw[dashed,red,thick] (-11,-1)--(-11,3);    
    \draw[dashed,red,thick] (-12,-1)--(-12,3);
    \end{tikzpicture}\quad \mapsto\quad  (1,1,2,3,1,2,2,1,1)$$    
Note that the number of compositions of $n$ with an odd number of parts is $2^{n-2}$.
\end{proof}

Finally, let 
\[Q_C:=\{\Lambda(\mcOv^{+}(2)), \Lambda(\mcOv^{-}(2))\}=\{45321, 54312\}\]
and define 
\[c^+_n:=\#\left\{\text{$\mcD$ on $[n]$}\ \Big| \ \substack{\displaystyle\mcD\ \text{is a fully supported, strongly connected}\\ \displaystyle \text{increasing chain and avoids}\ Q_C}\right\}.\]
\Cref{thm:overlaps} implies $\mcD$ avoids $Q_C$ if and only if $\mcD$ avoids 2-overlaps.

\begin{proposition}\label{prop:C_calc}
We have $\bar c_1^+=1$ and if $n\geq 2$, then $\bar c_n^+=2^{n-2}$.  It follows that
\[\sum_{n\geq 1}\bar c_n^+\, x^n= \frac{x-x^2}{1-2x}.\]

\end{proposition}

\begin{proof}
The fact that $\bar c_1^+=1$ is clear.  Let $\mcD=\{B_1,\ldots,B_k\}$ be a strongly connected increasing chain and let $\alpha_i:=|B_i|$.  If $\mcD$ avoids $Q_C$, then there are $1$-overlaps between all blocks in $\mcD$.  The map
\[\mcD\mapsto (\alpha_1-1,\ldots, \alpha_k-1)\] is a bijection between strongly connected, increasing chaings on $[n]$ with $k$ blocks that avoid $Q_C$ and $k$-compositions $(n-1)$.  Summing over $k$ gives $\bar c_n^+=2^{n-2}$.  For example:
$$\begin{tikzpicture}[scale=0.35,baseline={(0,0.3)}]
    \ppsansnumbers{15}{
    {0,0,0,0,0,0,0,0,1,1,8},
    {0,0,0,0,0,2,2,2,8,0,0},
    {0,0,3,3,3,8,0,0,0,0,0},
    {0,4,8,0,0,0,0,0,0,0,0},};
    \draw[dashed,red,thick] (-1,1)--(-1,5);
    \draw[dashed,red,thick] (-4,0)--(-4,4);
    \draw[dashed,red,thick] (-7,-1)--(-7,3);    
    \end{tikzpicture}\quad \mapsto\quad  (2,3,3,1)$$
\end{proof}

\begin{proof}[Proof of \Cref{thm:polished_GF,thm:semi-polished_GF12}]
Apply \Cref{thm:str_conn_2_all_SD} and \Cref{thm:connected2chains_gf} to \Cref{prop:P_calc}, \Cref{prop:B_calc}, and \Cref{prop:C_calc}.
\end{proof}

\subsection{Increasing chains and Calatan objects}
We conclude this section by showing how enumerating strongly connected increasing chains is linked to the enumeration of certain sub-collections of Catalan objects.  Recall that a Dyck path on $n$ is a path from $(0,0)$ to $(2n,0)$ corresponding to a sequence of ``up-steps" given by the vector $\langle1,1\rangle$ and ``down-steps" given by vector $\langle1,-1\rangle$ where in every initial segment there are at least as many up-steps as down-steps.  In other words, each Dyck path corresponds to a ballot word of $n$ $U$'s (for up) and $n$ $D$'s (for down).  It is well known that the number of Dyck paths on $n$ is the $n$-th Catalan number $C_n=\frac{1}{n+1}\binom{2n}{n}$.  

We associate a Dyck path to a fully supported, increasing chain on $[n]$ as follows.  Let $\mcD=\{B_1,\ldots,B_k\}$ and to each block $B_i$, we define the sequence 
\[M_i=U^{|B_i\setminus B_{i-1}|}D^{|B_i\setminus B_{i+1}|}\]
where $U^aD^b$ denotes $a$ up-steps, followed by $b$ down-steps.  Here we set $B_0=B_{k+1}=\varnothing$.  Define 
\[\DyckP(\mcD):=M_1\cdots M_k.\]
It is easy to check that $\DyckP(\mcD)$ forms a well-defined ballot word with exactly $n$ $U$'s and $n$ $D$'s and hence gives a Dyck path.

\begin{example}\label{ex:Dyck_path1}
Let $\mcD=\{[1,4]\prec[3,5]\prec[5,7]\}$.  Then
\[\DyckP(\mcD)=(U^4D^2)(UD^2)(U^2D^3)\]
is a Dyck path from $(0,0)$ to $(14,0)$.  Diagrammatically, we have
$$
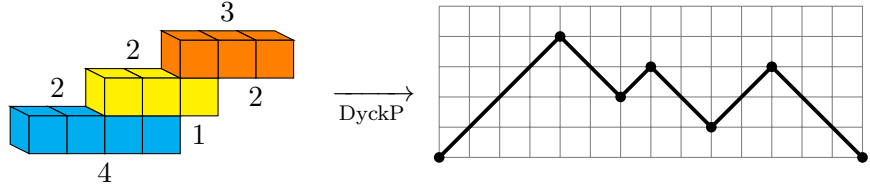
\begin{tikzpicture}[scale=0.5,baseline={(0,0.7)}]
            \ppsansnumbers{12}{ 
            {0,0,0,1,1,1,1},
            {0,0,3,3,3,0,0},
            {5,5,5,0,0,0,0}};
    \draw (-4,-0.5)node {$4$};
    \draw (-1.5,0.5)node {$1$};
    \draw (0,1.5)node {$2$};
    \draw (-5.25,1.75)node {$2$};
    \draw (-3.25,2.75)node {$2$};
    \draw (-0.75,3.75)node {$3$};
    \end{tikzpicture}\quad\xrightarrow[\DyckP]\qquad\
\begin{tikzpicture}[scale=0.4,baseline={(0,0.75)}]
\draw[step=1.0,black!50] (0,0) grid (14,5);
\fill (0,0) circle (5pt);
\fill (4,4) circle (5pt);
\fill (6,2) circle (5pt);
\fill (7,3) circle (5pt);
\fill (9,1) circle (5pt);
\fill (11,3) circle (5pt);
\fill (14,0) circle (5pt);
\draw[thick, line width = 0.5mm] (0,0) -- (4,4) -- (6,2) -- (7,3) -- (9,1) -- (11,3)-- (14,0);
\end{tikzpicture}$$
\end{example}

\noindent The following is proved in the proof of \cite[Proposition 4.1]{RS18}.

\begin{proposition}
The map $\DyckP$ is a bijection between fully supported staircase diagrams and Dyck paths on $n$.
\end{proposition}

\noindent Several combinatorial properties of staircase diagrams can be translated to properties of Dyck paths.  Given a Dyck path $P$, a \textbf{peak} is an occurrence of the sequence $UD$ in the assoicated ballot word of $P$, and a \textbf{valley} is an occurrence of $DU$.  We say the \textbf{height of the peak} in $P$ is the vertical distance of the peak from the $x$-axis in the path $P$.  We similarly define the \textbf{height of a valley} in $P$.  In \Cref{ex:Dyck_path1}, there are three peaks of respective heights 4,3,3, and two valleys of respective heights 2,1.  We leave the proof of the next lemma as an exercise.

\begin{lemma}\label{lem:Dyck_height_valleys}
    Let $\mcD=\{B_1,\ldots,B_k\}$ be a fully supported, increasing chain on $[n]$.  Then the following are true:
    \begin{enumerate}
        \item $\DyckP(\mcD)$ has $k$ peaks, one corresponding to each block $B_i$.  Moreover, the height of the peak associated to $B_i$ is $|B_i|$.
        \item $\DyckP(\mcD)$ has $k-1$ valleys, one corresponding to each pair $(B_i,B_{i+1})$.  Moreover, the height of the valley associated to $(B_i,B_{i+1})$ is $|B_i\cap B_{i+1}|$.
    \end{enumerate}
\end{lemma}

The next proposition follows directly from \Cref{lem:Dyck_height_valleys}.

\begin{proposition}\label{prop:counting_fs_inc_Catalan}
    Let $\mcD=\{B_1,\ldots,B_k\}$ be a fully supported, increasing chain on $[n]$.  Then $\mcD$ is strongly connected if and only if the height of each valley in $\DyckP(\mcD)$ is at least one.  

    As a consequence, the total number of fully supported, strongly connected, increasing chains on $[n]$ is the Catalan number $C_{n-1}$.
\end{proposition}

Note that \Cref{thm:smooth_enumeration} can be recovered by setting $Z^+(x)$ equal to the generating function for Catalan numbers and applying  \Cref{thm:str_conn_2_all_SD}, \Cref{thm:connected2chains_gf}.  We remark that the method of using staircase diagrams to compute the generating function for smooth permutation appears in \cite{RS17} and \cite{RS18}.

We give an outline of how to calculate the generating function of the number of smooth permutations that avoid $\Lambda(\mcCh_\ell^{\pm})$ for $\ell\geq 2$, and separately avoid $\Lambda(\mcOv_q^{\pm})$ for $q\geq 2$.  For $\ell\geq 2$, define \[Q_{B,\ell}:=\{\Lambda(\mcCh^{+}(\ell)), \Lambda(\mcCh^{-}(\ell))\}\]
and let 

\[b^+_{\ell,n}:=\#\left\{\text{$\mcD$ on $[n]$}\ \Big| \ \substack{\displaystyle\mcD\ \text{is a fully supported, strongly connected}\\ \displaystyle \text{increasing chain and avoids}\ Q_{B,\ell}}\right\}.\]

\begin{proposition}\label{prop:Narayana}
Fix $\ell\geq 2$ and let $\mcD=\{B_1,\ldots,B_k\}$ be a fully supported, strongly connected increasing chain on $[n]$.  Then $\mcD$ avoids $Q_{B,\ell}$ if and only if $\DyckP(\mcD)$ has at most $\ell-1$ peaks and the height of each valley is at least one.
In particular, \[b^+_{\ell,n}=\sum_{i=1}^{\ell-1} N(i,n-1)\]
where $\displaystyle N(n,i)=\frac{1}{i}\binom{n}{i}\binom{n}{i-1}$ is the $(n,i)$-Narayana number.
\end{proposition}

\begin{proof}
The Narayana number $N(n,i)$ counts the number of Dyck paths on $n$ with exactly $i$ peaks.  The proposition follows from \Cref{lem:Dyck_height_valleys} and \Cref{prop:counting_fs_inc_Catalan}.
\end{proof}

\begin{corollary}\label{cor:chain_rational}
    For $\ell\geq 2$, the generating function for the number of smooth permutations avoiding $Q_{B,\ell}$ is  rational.
\end{corollary}
\begin{proof}
    By \Cref{cor:GF_summary} and \Cref{prop:Narayana}, it suffices to show that the generating function $\sum_{n\geq i} N(i,n)\, x^n$ is rational for a fixed integer $i$.  However, it is known that 
    \[\sum_{n\geq i} N(n,i)\, x^n=\frac{x^i\sum_{j=1}^i N(i,j)\, x^{j-1}}{(1-x)^{2i-1}}.\]  
    This formula is stated in \cite[Example 40]{Ba06} (see also \cite[Lemma 3.3]{GAO2016}).  We remark that the sum in the numerator is called the $i$-th Narayana polynomial.
\end{proof}

For $q\geq 2$, define \[Q_{C,\ell}:=\{\Lambda(\mcOv^{+}(q)), \Lambda(\mcOv^{-}(q))\}\]
and let 

\[c^+_{q,n}:=\#\left\{\text{$\mcD$ on $[n]$}\ \Big| \ \substack{\displaystyle\mcD\ \text{is a fully supported, strongly connected}\\ \displaystyle \text{increasing chain and avoids}\ Q_{C,q}}\right\}.\]
\begin{proposition}\label{prop:bounded_valleys}
Fix $q\geq 2$ and let $\mcD=\{B_1,\ldots,B_k\}$ be a fully supported, strongly connected increasing chain on $[n]$.  Then $\mcD$ avoids $Q_{C,q}$ if and only if $\DyckP(\mcD)$ the height of each valley is at least one and less than $q$.  In particular,
\[c^+_{q,n}=f_{q-1,n-1}\]
where $f_{q,n}$ is the number of Dyck paths on $n$ whose valley heights are less than $q$.
\end{proposition}

\begin{proof}
The proposition follows from \Cref{lem:Dyck_height_valleys} and \Cref{prop:counting_fs_inc_Catalan}.
\end{proof}

\begin{corollary}\label{cor:overlap_rational}
    For $q\geq 2$, the generating function for the number of smooth permutations avoiding $Q_{C,q}$ is rational.
\end{corollary}

\begin{proof}
The corollary follows directly from techniques developed in \cite{EZ20}, but we include a proof for completeness.  Let $f_{q,n}$ denote the number of Dyck paths on $n$ whose valley heights are less than $q$ and define
\[F_q(x)=\sum_{n\geq 0} f_{q,n}\, x^n.\]
It suffices to show that $F_q(x)$ is rational for all $q\geq 2$.  First note that when $q=2$, then $F_2(x)$ is rational by \Cref{prop:bounded_valleys} and \Cref{prop:C_calc}.  We proceed by induction on $q$.  Let $P$ denote a Dyck path whose valley heights are less than $q$.  Split $P$ into $P_1$ and $P_2$ at the first occurrence, which $P$ returns to the $x$-axis.  We can write 
\[P=P_1P_2=(UP_1'D) P_2.\]
Note that $P_1'$ is Dyck path whose valley heights are less than $q-1$, while $P_2$ is a Dyck path whose valley heights are still less than $q$.  This gives 
\[f_{q,n}=\sum_{k=1}^{n-1}f_{q-1,k}\cdot f_{q,n-k-1}\] which implies
\[F_q(x)=1+xF_{q-1}(x)\cdot F_q(x)\quad \text{and}\quad F_q(x)=\frac{1}{1-xF_{q-1}(x)}.\]
By induction, $F_q(x)$ is rational for all $q\geq 2$.
\end{proof}

\bibliographystyle{amsalpha}
\bibliography{references}

@article{EZ20,
      title={Automatic Counting of Restricted Dyck Paths via (Numeric and Symbolic) Dynamic Programming}, 
      author={Shalosh B. Ekhad and Doron Zeilberger},
      year={2020},
      journal={preprint, arXiv.2006.01961},
      eprint={2006.01961},
      archivePrefix={arXiv},
      primaryClass={math.CO},
      url={https://arxiv.org/abs/2006.01961}, 
}

@article{GAO2016,
title = {Pattern-avoiding alternating words},
journal = {Discrete Applied Mathematics},
volume = {207},
pages = {56-66},
year = {2016},
issn = {0166-218X},
doi = {https://doi.org/10.1016/j.dam.2016.03.007},
url = {https://www.sciencedirect.com/science/article/pii/S0166218X16301275},
author = {Alice L.L. Gao and Sergey Kitaev and Philip B. Zhang},
}

@article {Ba06,
    AUTHOR = {Barry, Paul},
     TITLE = {On integer-sequence-based constructions of generalized
              {P}ascal triangles},
   JOURNAL = {J. Integer Seq.},
  FJOURNAL = {Journal of Integer Sequences},
    VOLUME = {9},
      YEAR = {2006},
    NUMBER = {2},
     PAGES = {Article 06.2.4, 34},
      ISSN = {1530-7638},
   MRCLASS = {11B83 (05A19 11B37 11B65)},
  MRNUMBER = {2217230},
MRREVIEWER = {Jau-Shyong\ Shiue},
}

@article {RS18,
    AUTHOR = {Richmond, Edward and Slofstra, William},
     TITLE = {Smooth {S}chubert varieties in the affine flag variety of type
              {$\tilde A$}},
   JOURNAL = {European J. Combin.},
  FJOURNAL = {European Journal of Combinatorics},
    VOLUME = {71},
      YEAR = {2018},
     PAGES = {125--138},
      ISSN = {0195-6698,1095-9971},
   MRCLASS = {14N15 (14M15)},
  MRNUMBER = {3802239},
MRREVIEWER = {Nickolas\ Jason\ Hein},
       DOI = {10.1016/j.ejc.2018.03.003},
       URL = {https://doi.org/10.1016/j.ejc.2018.03.003},
}

@article{Az23,
    title = {Divisor labelling of staircase diagrams and fiber bundle structures on Schubert varieties},
    author = {Azam,  Faqruddin Ali},
    journal = {Oklahoma State University Dissertations [11221]},
      YEAR = {2023},
    URL = {https://shareok.org/handle/11244/339048},
}

@article{BW26,
    title = {A Bijective Proof of an Unbalanced Wilf Equivalence},
    author = {Bridges, Jensen and Waite, Michael},
    journal = {preprint, arXiv:2608.21684},
    year= {2026}
}

@book {BB05,
    AUTHOR = {Bj\"orner, Anders and Brenti, Francesco},
     TITLE = {Combinatorics of {C}oxeter groups},
    SERIES = {Graduate Texts in Mathematics},
    VOLUME = {231},
 PUBLISHER = {Springer, New York},
      YEAR = {2005},
     PAGES = {xiv+363},
      ISBN = {978-3540-442387; 3-540-44238-3},
   MRCLASS = {05-01 (05E15 20F55)},
  MRNUMBER = {2133266},
MRREVIEWER = {Jian-yi\ Shi},
}

@article{Ha92,
    title = {Enumeration of Smooth {S}chubert Varieties},
    author = {Haiman, Mark},
    journal = {preprint},
    note = {unpublished},
    year= {1992}
}

@book {Ki11,
    AUTHOR = {Kitaev, Sergey},
     TITLE = {Patterns in permutations and words},
    SERIES = {Monographs in Theoretical Computer Science. An EATCS Series},
      NOTE = {With a foreword by Jeffrey B. Remmel},
 PUBLISHER = {Springer, Heidelberg},
      YEAR = {2011},
     PAGES = {xxii+494},
      ISBN = {978-3-642-17332-5; 978-3-642-17333-2},
   MRCLASS = {05-02 (05A05 05A15 68-02 68R05 68R15 68W32)},
  MRNUMBER = {3012380},
MRREVIEWER = {Sergi\ Elizalde},
       DOI = {10.1007/978-3-642-17333-2},
       URL = {https://doi.org/10.1007/978-3-642-17333-2},
}

@article{KS03,
title = {Finite transition matrices for permutations avoiding pairs of length four patterns},
journal = {Discrete Mathematics},
volume = {268},
number = {1},
pages = {171-183},
year = {2003},
issn = {0012-365X},
doi = {https://doi.org/10.1016/S0012-365X(03)00042-6},
url = {https://www.sciencedirect.com/science/article/pii/S0012365X03000426},
author = {Darla Kremer and Wai Chee Shiu}
}

@article {LS90,
    AUTHOR = {Lakshmibai, V. and Sandhya, B.},
     TITLE = {Criterion for smoothness of {S}chubert varieties in {${\rm
              Sl}(n)/B$}},
   JOURNAL = {Proc. Indian Acad. Sci. Math. Sci.},
  FJOURNAL = {Indian Academy of Sciences. Proceedings. Mathematical
              Sciences},
    VOLUME = {100},
      YEAR = {1990},
    NUMBER = {1},
     PAGES = {45--52},
      ISSN = {0253-4142,0973-7685},
   MRCLASS = {14M15 (14L35)},
  MRNUMBER = {1051089},
MRREVIEWER = {H.\ H.\ Andersen},
       DOI = {10.1007/BF02881113},
       URL = {https://doi.org/10.1007/BF02881113},
}

@incollection {Ca94,
    AUTHOR = {Carrell, James B.},
     TITLE = {The {B}ruhat graph of a {C}oxeter group, a conjecture of
              {D}eodhar, and rational smoothness of {S}chubert varieties},
 BOOKTITLE = {Algebraic groups and their generalizations: classical methods
              ({U}niversity {P}ark, {PA}, 1991)},
    SERIES = {Proc. Sympos. Pure Math.},
    VOLUME = {56, Part 1},
     PAGES = {53--61},
 PUBLISHER = {Amer. Math. Soc., Providence, RI},
      YEAR = {1994},
      ISBN = {0-8218-1540-7},
   MRCLASS = {14M15 (14L35)},
  MRNUMBER = {1278700},
MRREVIEWER = {E.\ Aky\i ld\i z},
       DOI = {10.1090/pspum/056.1/1278700},
       URL = {https://doi.org/10.1090/pspum/056.1/1278700},
}

@article {Ry87,
    AUTHOR = {Ryan, Kevin M.},
     TITLE = {On {S}chubert varieties in the flag manifold of {${\rm
              Sl}(n,{\bf C})$}},
   JOURNAL = {Math. Ann.},
  FJOURNAL = {Mathematische Annalen},
    VOLUME = {276},
      YEAR = {1987},
    NUMBER = {2},
     PAGES = {205--224},
      ISSN = {0025-5831,1432-1807},
   MRCLASS = {14M15 (14M17)},
  MRNUMBER = {870962},
MRREVIEWER = {Konrad\ Drechsler},
       DOI = {10.1007/BF01450738},
       URL = {https://doi.org/10.1007/BF01450738},
}

@article {BMB07,
    AUTHOR = {Bousquet-M\'elou, Mireille and Butler, Steve},
     TITLE = {Forest-like permutations},
   JOURNAL = {Ann. Comb.},
  FJOURNAL = {Annals of Combinatorics},
    VOLUME = {11},
      YEAR = {2007},
    NUMBER = {3-4},
     PAGES = {335--354},
      ISSN = {0218-0006,0219-3094},
   MRCLASS = {05A05 (05A15)},
  MRNUMBER = {2376109},
MRREVIEWER = {Daniel\ E.\ Warren},
       DOI = {10.1007/s00026-007-0322-1},
       URL = {https://doi.org/10.1007/s00026-007-0322-1},
}

@article {RS17,
    AUTHOR = {Richmond, Edward and Slofstra, William},
     TITLE = {Staircase diagrams and enumeration of smooth {S}chubert
              varieties},
   JOURNAL = {J. Combin. Theory Ser. A},
  FJOURNAL = {Journal of Combinatorial Theory. Series A},
    VOLUME = {150},
      YEAR = {2017},
     PAGES = {328--376},
      ISSN = {0097-3165,1096-0899},
   MRCLASS = {14M15 (05E18)},
  MRNUMBER = {3645580},
MRREVIEWER = {Peter\ Crooks},
       DOI = {10.1016/j.jcta.2017.03.009},
       URL = {https://doi.org/10.1016/j.jcta.2017.03.009},
}

@article {GG20,
    AUTHOR = {Gaetz, Christian and Gao, Yibo},
     TITLE = {Self-dual intervals in the {B}ruhat order},
   JOURNAL = {Selecta Math. (N.S.)},
  FJOURNAL = {Selecta Mathematica. New Series},
    VOLUME = {26},
      YEAR = {2020},
    NUMBER = {5},
     PAGES = {Paper No. 77, 23},
      ISSN = {1022-1824,1420-9020},
   MRCLASS = {05E99 (14M15)},
  MRNUMBER = {4177574},
MRREVIEWER = {Marko\ Radovanovi\'c},
       DOI = {10.1007/s00029-020-00608-z},
       URL = {https://doi.org/10.1007/s00029-020-00608-z},
}
\end{document}